\documentclass[a4paper,12pt]{article}
\usepackage[cp1251]{inputenc}
\usepackage[russian]{babel}
\usepackage{amsfonts, amssymb, amsmath, amsthm, amscd}
\usepackage{cite}
\author{A.A. Vasil'eva}
\title{Entropy numbers of embedding operators of weighted Sobolev spaces
with weights that are functions of distance from some $h$-set}
\date{}
\begin{document}

\maketitle

\newenvironment{Biblio}{%
                  \renewcommand{\refname}{\footnotesize REFERENCES}%
                  \begin{thebibliography}{99}%
                  \itemsep -0.2em}%
                  {\end{thebibliography}}

\def\inff{\mathop{\smash\inf\vphantom\sup}}
\renewcommand{\le}{\leqslant}
\renewcommand{\ge}{\geqslant}
\newcommand{\sgn}{\mathrm {sgn}\,}
\newcommand{\inter}{\mathrm {int}\,}
\newcommand{\dist}{\mathrm {dist}}
\newcommand{\supp}{\mathrm {supp}\,}
\newcommand{\R}{\mathbb{R}}
\renewcommand{\C}{\mathbb{C}}
\newcommand{\Z}{\mathbb{Z}}
\newcommand{\N}{\mathbb{N}}
\newcommand{\Q}{\mathbb{Q}}
\theoremstyle{plain}
\newtheorem{Trm}{Theorem}
\newtheorem{trma}{Theorem}

\newtheorem{Def}{Definition}
\newtheorem{Cor}{Corollary}
\newtheorem{Lem}{Lemma}
\newtheorem{Rem}{Remark}
\newtheorem{Sta}{Proposision}
\newtheorem{Sup}{Assumption}
\newtheorem{Exa}{Example}
\renewcommand{\proofname}{\bf Proof}
\renewcommand{\thetrma}{\Alph{trma}}

\section{Introduction}

In this paper we obtain order estimates for entropy numbers of
embedding operator of weighted Sobolev spaces on a John domain
into weighted Lebesgue space. Estimates for $n$-widths of such
embeddings were recently obtained in \cite{vas_width_raspr,
vas_w_lim}.

\begin{Def}
Let $X$, $Y$ be normed spaces, and let $T:X\rightarrow Y$ be a
linear continuous operator. Entropy numbers of $T$ are defined by
$$
e_k(T)=\inf \left\{\varepsilon>0:\; \exists y_1, \, \dots, \,
y_{2^{k-1}}\in Y: \; T(B_X) \subset \cup_{i=1}^{2^{k-1}}
(y_i+\varepsilon B_Y) \right\}, \quad k\in \N.
$$
\end{Def}

For properties of entropy numbers, we refer the reader to the
books \cite{piet_op, carl_steph, edm_trieb_book}. Kolmogorov,
Tikhomirov, Birman and Solomyak \cite{kolm_tikh1, tikh_entr, birm}
studied properties of $\varepsilon$-entropy (this magnitude is
related to entropy numbers of embedding operators).

Estimates for entropy numbers of the embedding operator of $l_p^m$
into $l_q^m$ were obtained in the paper of Sch\"{u}tt
\cite{c_schutt} (see also \cite{edm_trieb_book}). Here $l_p^m$
$(1\le p\le \infty)$ is the space $\R^m$ with the norm
$$
\|(x_1, \, \dots , \, x_m)\| _q\equiv\|(x_1, \, \dots , \, x_m)\|
_{l_p^m}= \left\{
\begin{array}{l}(| x_1 | ^p+\dots+ | x_m | ^p)^{1/p},\text{ if
}p<\infty ,\\ \max \{| x_1 | , \, \dots, \, | x_m |\},\text{ if
}p=\infty .\end{array}\right .
$$
 Later Edmunds
and Netrusov \cite{edm_netr1}, \cite{edm_netr2} generalized this
result for vector-valued sequence spaces (in particular, for
sequence spaces with mixed norm). Haroske, Triebel, K\"{u}hn,
Leopold, Sickel, Skrzypczak \cite{kuhn4, kuhn5, kuhn_01,
kuhn_01_g, kuhn_05, kuhn_08, kuhn_leopold, kuhn_tr_mian, har94_1,
har94_2, haroske, haroske2, haroske3, har_tr05} studied the
problem of estimating entropy numbers of embeddings of weighted
sequence spaces or weighted Besov and Triebel--Lizorkin spaces.

Lifshits and Linde \cite{lif_linde} obtained estimates for entropy
numbers of two-weighted Hardy-type operators on a semiaxis (under
some conditions on weights). The similar problem for one-weighted
Riemann-Liouville operators was considered in the paper of
Lomakina and Stepanov \cite{step_lom}. In addition, Lifshits and
Linde \cite{lifs_m, l_l, l_l1} studied the problem of estimating
entropy numbers of two-weighted summation operators on a tree.

Triebel \cite{tr_jat} and Mieth \cite{mieth_15} studied the
problem of estimating entropy numbers of embedding operators of
weighted Sobolev spaces on a ball with weights that have
singularity at the origin.

Estimates of entropy numbers of weighted function spaces are
applied in spectral theory of some degenerate elliptic operators
(see, e.g., \cite{har94_1, har94_2, haroske4, har_tr05, kuhn4,
kuhn5, kuhn_leopold}) and in estimating the probability of small
deviation of Gaussian random functions (see, e.g.,
\cite{lif_linde, l_l, li_linde}).

The paper is organized as follows. In this section we introduce
notations and some basic definitions, and we conclude this section
with main result about estimates for entropy numbers of embeddings
of weighted Sobolev spaces. In \S 2 we formulate some known
results which will be required in the sequel. In \S 3 we obtain
upper estimates for entropy numbers of embedding operators of some
function spaces on a set with tree-like structure (the similar
results for $n$-widths are obtained in \cite{vas_width_raspr}). In
\S 4 we prove Theorems \ref{th1}, \ref{th2} and \ref{th3} about
estimates for entropy numbers of embedding operators of weighted
Sobolev spaces. In \S 5 we obtain estimates of entropy numbers of
two-weighted summation operators on a tree.

Let us give the definition of a John domain.

We denote by $AC[t_0, \, t_1]$ the space of absolutely continuous
functions on an interval $[t_0, \, t_1]$.

Let $B_a(x)$ be the closed euclidean ball of radius $a$ in $\R^d$
centered at the point $x$.
\begin{Def}
\label{fca} Let $\Omega\subset\R^d$ be a bounded domain, and let
$a>0$. We say that $\Omega \in {\bf FC}(a)$ if there exists a
point $x_*\in \Omega$ such that for any $x\in \Omega$ there exist
$T(x)>0$ and a curve $\gamma _x:[0, \, T(x)] \rightarrow\Omega$
with the following properties:
\begin{enumerate}
\item $\gamma _x\in AC[0, \, T(x)]$, $\left|\frac{d \gamma _x(t)}{dt}\right|=1$ a.e.,
\item $\gamma _x(0)=x$, $\gamma _x(T(x))=x_*$,
\item $B_{at}(\gamma _x(t))\subset \Omega$ for any $t\in [0, \, T(x)]$.
\end{enumerate}
\end{Def}

\smallskip

\begin{Def}
We say that $\Omega$ satisfies the John condition (and call
$\Omega$ a John domain) if $\Omega\in {\bf FC}(a)$ for some $a>0$.
\end{Def}

For a bounded domain the John condition is equivalent to the
flexible cone condition (see the definition in \cite{besov_il1}).
As examples of such domains we can take \begin{enumerate}
\item domains with Lipschitz boundary; \item the interior of the Koch
snowflake; \item domains $\Omega=\cup _{0\le t\le T}
B_{ct}(\gamma(t))$, where $\gamma:[0, \, T] \rightarrow \R^d$ is a
curve with natural parametrization and $c>0$. \end{enumerate}
Domains with zero inner angles do not satisfy the John condition.

Reshetnyak \cite{resh1, resh2} found the integral representation
for smooth functions defined on a John domain $\Omega$ in terms of
their derivatives of order $r$. It follows from this integral
representation that, for $p>1$, $1\le q<\infty$ and $\frac
rd+\frac 1q-\frac 1p\ge 0$ ($\frac rd+\frac 1q-\frac 1p>0$,
respectively) the class $W^r_p(\Omega)$ is continuously
(respectively, compactly) embedded in the space $L_q(\Omega)$
(i.e., the conditions of the continuous and compact embedding are
the same as for $\Omega=[0, \, 1]^d$).

Introduce the notion of $h$-set according to \cite{m_bricchi1}.

Denote by $\mathbb{H}$ the set of all nondecreasing positive
functions defined on $(0, \, 1]$.
\begin{Def}
\label{h_set} Let $\Gamma\subset \mathbb{R}^d$ be a non-empty
compact set, and let $h\in \mathbb{H}$. We say that $\Gamma$ is an
$h$-set if there are a constant $c_*\ge 1$ and a finite countably
additive measure $\mu$ on $\mathbb{R}^d$ such that ${\mathrm
{supp}\,} \mu=\Gamma$ and for any $x\in \Gamma$, $t\in (0, \, 1]$
\begin{align}
\label{c1htmu} c_*^{-1}h(t)\le \mu(B_t(x))\le c_* h(t).
\end{align}
\end{Def}

\begin{Exa}
Let $\Gamma \subset \R^d$ be a Lipschitz manifold of dimension
$k$, $0\le k<d$. Then $\Gamma$ is an $h$-set with $h(t)=t^k$.
\end{Exa}

\begin{Exa}
Let $\Gamma \subset \R^2$ the the Koch snowflake. Then $\Gamma$ is
an $h$-set with $h(t)=t^{\log 4/\log 3}$ (see \cite[p.
66--68]{p_mattila}).
\end{Exa}

Let us formulate the main result of this paper.

Everywhere below, we use the notation $\log x=\log_2 x$.

Let $|\cdot|$ be a norm on $\mathbb{R}^d$, and let $E, \,
E'\subset \mathbb{R}^d$, $x\in \mathbb{R}^d$. We set
$$
{\rm diam}_{|\cdot|}\, E=\sup \{|y-z|:\; y, \, z\in E\}, \;\; {\rm
dist}_{|\cdot|}\, (x, \, E)=\inf \{|x-y|:\; y\in E\}.
$$

Let $\Omega\in {\bf FC}(a)$ be a bounded domain, and let
$\Gamma\subset \partial \Omega$ be an $h$-set. Further we suppose
that in some neighborhood of zero the function $h\in \mathbb{H}$
is defined by
\begin{align}
\label{def_h} h(t)=t^{\theta}|\log t|^{\gamma}\tau(|\log t|),
\;\;\; 0\le \theta<d,
\end{align}
where $\tau:(0, \, +\infty)\rightarrow (0, \, +\infty)$ is an
absolutely continuous function such that
\begin{align}
\label{yty} \frac{t\tau'(t)}{\tau(t)} \underset{t\to+\infty}{\to}
0.
\end{align}

Let $1<p\le \infty$, $1\le q<\infty$, $r\in \N$, $\delta:=r+\frac
dq-\frac dp>0$, $\beta_g$, $\beta_v\in \R$, $g(x)=\varphi_g({\rm
dist}_{|\cdot|}(x, \, \Gamma))$, $v(x)=\varphi_v({\rm
dist}_{|\cdot|}(x, \, \Gamma))$,
\begin{align}
\label{ghi_g0} \varphi_g(t)=t^{-\beta_g}|\log
t|^{-\alpha_g}\rho_g(|\log t|), \;\; \varphi_v(t)=t^{-\beta_v}
|\log t|^{-\alpha_v}\rho_v(|\log t|),
\end{align}
where $\rho_g$ and $\rho_v$ are absolutely continuous functions,
\begin{align}
\label{psi_cond} \frac{t\rho'_g(t)}{\rho_g(t)}
\underset{t\to+\infty}{\to}0, \;\; \frac{t\rho'_v(t)}{\rho_v(t)}
\underset{t\to+\infty}{\to}0.
\end{align}

In addition, we suppose that
\begin{align}
\label{muck} \beta_v<\frac{d-\theta}{q} \quad \text{or} \quad
\beta_v=\frac{d-\theta}{q}, \quad \alpha_v>\frac{1-\gamma}{q}.
\end{align}

Without loss of generality we may assume that
$\overline{\Omega}\subset \left(-\frac 12, \, \frac 12\right)^d$.

We set $\beta=\beta_g+\beta_v$, $\alpha=\alpha_g+\alpha_v$,
$\rho(y)=\rho_g(y)\rho_v(y)$, $\mathfrak{Z}=(r,\, d, \, p, \, q,
\, g, \, v, \, h, \, a, \, c_*)$, $\mathfrak{Z}_*=(\mathfrak{Z},
\, R)$, where $c_*$ is the constant from Definition \ref{h_set}
and \label{r_def}$R={\rm diam}\, \Omega$.

We use the following notations for order inequalities. Let $X$,
$Y$ be sets, and let $f_1$, $f_2:\ X\times Y\rightarrow
\mathbb{R}_+$. We write $f_1(x, \, y)\underset{y}{\lesssim} f_2(x,
\, y)$ (or $f_2(x, \, y)\underset{y}{\gtrsim} f_1(x, \, y)$) if
for any $y\in Y$ there exists $c(y)>0$ such that $f_1(x, \, y)\le
c(y)f_2(x, \, y)$ for any $x\in X$; $f_1(x, \,
y)\underset{y}{\asymp} f_2(x, \, y)$ if $f_1(x, \, y)
\underset{y}{\lesssim} f_2(x, \, y)$ and $f_2(x, \,
y)\underset{y}{\lesssim} f_1(x, \, y)$.

Denote by ${\cal P}_{r-1}(\R^d)$ the space of polynomials on
$\R^d$ of degree not exceeding $r-1$. For a measurable set
$E\subset \R^d$ we set
$${\cal P}_{r-1}(E)= \{f|_E:\, f\in {\cal P}_{r-1}(\R^d)\}.$$
Notice that $W^r_{p,g}(\Omega) \supset {\cal P}_{r-1}(\Omega)$.

In Theorems \ref{th1}, \ref{th2}, \ref{th3} the conditions on
weights are such that $W^r_{p,g}(\Omega) \subset L_{q,v}(\Omega)$
and there exist $M>0$ and a linear continuous operator
$P:L_{q,v}(\Omega) \rightarrow {\cal P}_{r-1}(\Omega)$ such that
for any function $f\in W^r_{p,g}(\Omega)$
\begin{align}
\label{fpflqv} \|f-Pf\|_{L_{q,v}(\Omega)} \le M
\left\|\frac{\nabla^r f}{g}\right\| _{L_p(\Omega)}
\end{align}
(see \cite{vas_vl_raspr, vas_vl_raspr2, vas_sib, vas_w_lim}).

\begin{Rem}
\label{peq0} Let $x_*$ be the point from Definition \ref{fca}, and
let $R' ={\rm dist}_{|\cdot|} (x_*, \, \partial \Omega)$. The
operator $P$ defined in \cite{vas_vl_raspr2} (see also
\cite{vas_sing}) has the following property: there exists
$s_0=s_0(\mathfrak{Z})\in (0, \, 1)$ such that, for any function
$f\in C^\infty(\Omega) \cap L_{q,v}(\Omega)$ satisfying the
condition $f|_{B_{s_0R'}(x_*)}=0$, the equality $Pf=0$ holds.
\end{Rem}

We set $\hat W^r_{p,g}(\Omega) =\{f-Pf:\; f\in
W^r_{p,g}(\Omega)\}$. Let $\hat {\cal W}^r_{p,g}(\Omega) ={\rm
span}\, \hat W^r_{p,g}(\Omega)$ be equipped with norm $\|f\|_{\hat
{\cal W}^r_{p,g}(\Omega)}:= \left\|\frac{\nabla^r f}{g}\right\|
_{L_p(\Omega)}$. Denote by $I:\hat {\cal W}^r_{p,g}(\Omega)
\rightarrow L_{q,v}(\Omega)$ the embedding operator. From
(\ref{fpflqv}) it follows that $I$ is continuous.

\begin{Trm} \label{th1}
Let (\ref{def_h}), (\ref{yty}), (\ref{ghi_g0}), (\ref{psi_cond}),
(\ref{muck}) hold and $0<\theta<d$.
\begin{enumerate}
\item Suppose that $\beta-\delta<-\theta\left(\frac 1q-\frac
1p\right)_+$. We set $\alpha_0=\alpha$ for
$\beta_v<\frac{d-\theta}{q}$ and $\alpha_0=\alpha-\frac 1q$ for
$\beta_v=\frac{d-\theta}{q}$. We also suppose that
$\frac{\delta}{d}\ne \frac{\delta-\beta}{\theta}$. Denote
$\sigma_*(n)=1$ for $\frac{\delta}{d}<
\frac{\delta-\beta}{\theta}$ and
$$\sigma_*(n)=(\log
n)^{-\alpha_0+\frac{(\beta-\delta)\gamma}{\theta}} \rho(\log n)
\tau^{\frac{\beta-\delta}{\theta}}(\log n)$$ for
$\frac{\delta}{d}> \frac{\delta-\beta}{\theta}$. Then
$$
e_n(I:\hat {\cal W}^r_{p,g}(\Omega) \rightarrow L_{q,v}(\Omega))
\underset{\mathfrak{Z}_*}{\asymp} n^{-\min
\left\{\frac{\delta}{d}, \, \frac{\delta-\beta}{\theta}\right\}
+\frac 1q-\frac 1p} \sigma_*(n).
$$
\item Suppose that $\beta-\delta=-\theta\left(\frac 1q-\frac
1p\right)_+$.
\begin{enumerate}
\item Let $p\ge q$ and $\alpha_0:=\alpha-(1-\gamma)\left(\frac 1q-\frac
1p\right)>0$ for $\beta_v<\frac{d-\theta}{q}$,
$\alpha_0:=\alpha-1-(1-\gamma)\left(\frac 1q-\frac 1p\right)>0$
for $\beta_v=\frac{d-\theta}{q}$. Then
$$
e_n(I:\hat {\cal W}^r_{p,g}(\Omega) \rightarrow L_{q,v}(\Omega))
\underset{\mathfrak{Z}_*}{\asymp} (\log n)^{-\alpha_0} \rho(\log
n) \tau^{-\frac 1q+\frac 1p}(\log n).
$$
\item Let $p<q$ and $\alpha_0:=\alpha>0$ for
$\beta_v<\frac{d-\theta}{q}$, $\alpha_0:=\alpha-\frac 1q>0$ for
$\beta_v=\frac{d-\theta}{q}$. Suppose that $\alpha_0\ne \frac
1p-\frac 1q$. Then
$$
e_n(I:\hat {\cal W}^r_{p,g}(\Omega) \rightarrow L_{q,v}(\Omega))
\underset{\mathfrak{Z}_*}{\asymp} n^{\frac 1q-\frac 1p} (\log
n)^{-\alpha_0-\frac 1q+\frac 1p} \rho(\log n)
$$
for $\alpha_0>\frac 1p-\frac 1q$,
$$
e_n(I:\hat {\cal W}^r_{p,g}(\Omega) \rightarrow L_{q,v}(\Omega))
\underset{\mathfrak{Z}_*}{\asymp} n^{-\alpha_0} \rho(n)
$$
for $\alpha_0<\frac 1p-\frac 1q$.
\end{enumerate}
\end{enumerate}
\end{Trm}

Now we consider the case $\theta=0$.
\begin{Trm} \label{th2}
Let (\ref{def_h}), (\ref{yty}), (\ref{ghi_g0}), (\ref{psi_cond})
hold and $\theta=0$, $\beta-\delta<0$, $\beta_v<\frac{d}{q}$. Then
$$
e_n(I:\hat {\cal W}^r_{p,g}(\Omega) \rightarrow L_{q,v}(\Omega))
\underset{\mathfrak{Z}_*}{\asymp} n^{-\frac rd}.
$$
\end{Trm}
In the case $\theta=0$, $\beta-\delta=0$ we suppose that
$\rho_g(t)=|\log t|^{-\lambda_g}$, $\rho_v=|\log t|^{-\lambda_v}$,
$\tau(t)=|\log t|^{\nu}$ (in the general case the estimates in
assertion 1 of Theorem \ref{th3} can be obtained similarly).
Denote $\lambda =\lambda_g+\lambda_v$.

\begin{Trm} \label{th3}
Suppose that (\ref{def_h}), (\ref{yty}), (\ref{ghi_g0}),
(\ref{psi_cond}) hold and $\theta=0$, $\beta-\delta=0$,
$\beta_v<\frac{d}{q}$.
\begin{enumerate}
\item Let $\alpha-(1-\gamma)\left(\frac 1q-\frac 1p\right)_+>0$.
Suppose that $\frac{\alpha}{1-\gamma}\ne \frac{\delta}{d}$. We set
$\sigma_*(n)=1$ for $\frac{\delta}{d}<\frac{\alpha}{1-\gamma}$ and
$\sigma_*(n)= (\log n)^{-\lambda-\frac{\alpha\nu}{1-\gamma}}$ for
$\frac{\delta}{d}>\frac{\alpha}{1-\gamma}$. Then
$$
e_n(I:\hat {\cal W}^r_{p,g}(\Omega) \rightarrow L_{q,v}(\Omega))
\underset{\mathfrak{Z}_*}{\asymp} n^{-\min
\left\{\frac{\delta}{d}, \, \frac{\alpha}{1-\gamma}\right\}+\frac
1q-\frac 1p} \sigma_*(n).
$$
\item Suppose that $\alpha-(1-\gamma)\left(\frac 1q-\frac 1p\right)_+=0$,
$\lambda>(1-\nu)\left(\frac 1q-\frac 1p\right)_+$.
\begin{enumerate}
\item Let $p\ge q$. Then
$$
e_n(I:\hat {\cal W}^r_{p,g}(\Omega) \rightarrow L_{q,v}(\Omega))
\underset{\mathfrak{Z}_*}{\asymp} (\log
n)^{-\lambda+(1-\nu)\left(\frac 1q-\frac 1p\right)}.
$$
\item Let $p<q$, $\lambda\ne \frac 1p-\frac 1q$. Then
$$
e_n(I:\hat {\cal W}^r_{p,g}(\Omega) \rightarrow L_{q,v}(\Omega))
\underset{\mathfrak{Z}_*}{\asymp} n^{\frac 1q-\frac 1p} (\log
n)^{-\lambda+\frac 1p-\frac 1q}
$$
for $\lambda>\frac 1p-\frac 1p$,
$$
e_n(I:\hat {\cal W}^r_{p,g}(\Omega) \rightarrow L_{q,v}(\Omega))
\underset{\mathfrak{Z}_*}{\asymp} n^{-\lambda}
$$
for $\lambda<\frac 1p-\frac 1q$.
\end{enumerate}
\end{enumerate}
\end{Trm}
If $\Gamma$ is a singleton, then estimates of entropy numbers are
given by formulas from Theorem \ref{th3} with $\gamma=0$,
$\tau\equiv 1$. The proof is the same as for Theorem \ref{th3}.
These estimates are the generalization of the result of Triebel
\cite{tr_jat} (in \cite{tr_jat} the case $p=q$ was considered).

Without loss of generality we may assume that $|(x_1, \, \dots, \,
x_d)|=\max _{1\le i\le d}|x_i|$. Further we shall denote ${\rm
dist}:={\rm dist}_{|\cdot|}$, ${\rm diam}:={\rm diam}_{|\cdot|}$.

\section{Preliminaries}

The following properties of entropy numbers are well-known (see,
e.g., \cite{edm_trieb_book}):
\begin{enumerate}
\item if $T:X\rightarrow Y$, $S:Y\rightarrow Z$ are linear
continuous operators, then $e_{k+l-1}(ST)\le e_k(S)e_l(T)$;
\item if $T, \, S:X\rightarrow Y$ are linear continuous operators,
then
\begin{align}
\label{eklspt} e_{k+l-1}(S+T)\le e_k(S)+e_l(T).
\end{align}
\end{enumerate}
From property 1 it follows that
\begin{align}
\label{mult_n} e_k(ST)\le \|S\|e_k(T), \quad e_k(ST)\le
\|T\|e_k(S).
\end{align}

Further we denote by $I_\nu$ the identity operator on $\R^\nu$.

\begin{trma}
\label{shutt_trm} {\rm (see \cite{c_schutt, edm_trieb_book}).} Let
$1\le p\le q\le \infty$. Then
$$
e_k(I_\nu:l_p^{\nu}\rightarrow l_q^{\nu}) \underset{p,q}{\asymp}
\left\{ \begin{array}{l} 1, \quad 1\le k\le
\log \nu, \\
\left(\frac{\log\left(1+\frac{\nu}{k}\right)}{k}\right)^{\frac1p
-\frac 1q}, \quad \log \nu\le k\le \nu, \\
2^{-\frac{k}{\nu}}\nu^{\frac 1q-\frac 1p}, \quad \nu\le
k.\end{array}\right.
$$
Let $1\le q<p\le \infty$. Then
$$
e_k(I_\nu:l_p^{\nu}\rightarrow l_q^{\nu}) \underset{p,q}{\asymp}
2^{-\frac{k}{\nu}}\nu^{\frac 1q-\frac 1p}, \quad k\in \N.
$$
\end{trma}
\begin{Rem}
In estimates from \cite{edm_trieb_book} the value $2\nu$ was taken
instead of $\nu$ since the spaces $l_p^\nu$, $l_q^\nu$ were
considered as spaces over $\C$.
\end{Rem}

In the paper of K\"{u}hn \cite{kuhn_08} the order estimates for
entropy numbers od diagonal operators $D_\sigma: l_p \rightarrow
l_q$ were obtained for $p>q$.

\begin{trma}
\label{kuhn_trm} {\rm (see \cite{kuhn_08}).} Let $0<q<p\le
\infty$, $\sigma =(\sigma_k)_{k\in \N} \in l_{\frac{pq}{p-q}}$,
$\omega_n =\left(\sum \limits _{k=n}^\infty
\sigma_k^{\frac{pq}{p-q}} \right)^{\frac 1q-\frac 1p}$. Suppose
that there exists $C>0$ such that $\omega_n \le C \omega_{2n}$ for
any $n$. We define the operator $D_\sigma:l_p \rightarrow l_q$ by
$D_\sigma(x_k)_{k\in \N} =(\sigma_kx_k)_{k\in \N}$. Then
$e_n(D_\sigma:l_p \rightarrow l_q) \underset{C,p,q}{\asymp}
\omega_n$.
\end{trma}

The following result was proved by Lifshits \cite{lifs_m}.
\begin{trma}
\label{lifs_sta} {\rm (see \cite{lifs_m}).} Let $X$, $Y$ be normed
spaces, and let $V\in L(X, \, Y)$, $\{V_\nu\}_{\nu\in {\cal
N}}\subset L(X, \, Y)$. Then for any $n\in \N$
$$
e_{n+[\log_2|{\cal N}|]+1}(V)\le \sup _{\nu\in {\cal N}}
e_n(V_\nu)+\sup _{x\in B_X} \inf _{\nu\in {\cal N}} \|Vx-V_\nu
x\|_Y.
$$
\end{trma}

\section{Estimates for entropy numbers of function classes on a set with tree-like structure}

First we give some notations.

Let $(\Omega, \, \Sigma, \, {\rm mes})$ be a measure space. We say
that sets $A$, $B\subset \Omega$ are disjoint if ${\rm mes}(A\cap
B)=0$. Let $E$, $E_1, \, \dots, \, E_m\subset \Omega$ be
measurable sets, and let $m\in \N\cup \{\infty\}$. We say that
$\{E_i\}_{i=1}^m$ is a partition of $E$ if the sets $E_i$ are
pairwise disjoint and ${\rm mes}\left(\left(\cup _{i=1}^m
E_i\right)\bigtriangleup E\right)=0$.

Denote by $\chi_E(\cdot)$ the indicator function of a set $E$.

Let ${\cal G}$ be a graph containing at most countable number of
vertices. We shall denote by ${\bf V}({\cal G})$ and by ${\bf
E}({\cal G})$ the set of vertices and the set of edges of ${\cal
G}$, respectively. Two vertices are called {\it adjacent} if there
is an edge between them. Let $\xi_i\in {\bf V}({\cal G})$, $1\le
i\le n$. The sequence $(\xi_1, \, \dots, \, \xi_n)$ is called a
{\it path} if the vertices $\xi_i$ and $\xi_{i+1}$ are adjacent
for any $i=1, \, \dots , \, n-1$. If all the vertices $\xi_i$ are
distinct, then such a path is called {\it simple}.

Let $({\cal T}, \, \xi_0)$ be a tree with a distinguished vertex
(or a root) $\xi_0$. We introduce a partial order on ${\bf
V}({\cal T})$ as follows: we say that $\xi'>\xi$ if there exists a
simple path $(\xi_0, \, \xi_1, \, \dots , \, \xi_n, \, \xi')$ such
that $\xi=\xi_k$ for some $k\in \overline{0, \, n}$. In this case,
we set $\rho_{{\cal T}}(\xi, \, \xi')=\rho_{{\cal T}}(\xi', \,
\xi) =n+1-k$. In addition, we denote $\rho_{{\cal T}}(\xi, \,
\xi)=0$. If $\xi'>\xi$ or $\xi'=\xi$, then we write $\xi'\ge \xi$.
This partial order on ${\cal T}$ induces a partial order on its
subtree.

Given $j\in \Z_+$, $\xi\in {\bf V}({\cal T})$, we denote
$$
\label{v1v}{\bf V}_j(\xi):={\bf V}_j ^{{\cal T}}(\xi):=
\{\xi'\ge\xi:\; \rho_{{\cal T}}(\xi, \, \xi')=j\}.
$$
For $\xi\in {\bf V}({\cal T})$ we denote by ${\cal T}_\xi=({\cal
T}_\xi, \, \xi)$ the subtree in ${\cal T}$ with vertex set
\begin{align}
\label{vpvtvpv} \{\xi'\in {\bf V}({\cal T}):\xi'\ge \xi\}.
\end{align}

Let ${\cal G}$ be a subgraph in ${\cal T}$. Denote by ${\bf
V}_{\max} ({\cal G})$ and ${\bf V}_{\min}({\cal G})$ the sets of
maximal and minimal vertices in ${\cal G}$, respectively.

Let ${\bf W}\subset {\bf V}({\cal T})$. We say that ${\cal
G}\subset {\cal T}$ is a maximal subgraph on the set of vertices
${\bf W}$ if ${\bf V}({\cal G})={\bf W}$ and any two vertices
$\xi'$, $\xi''\in {\bf W}$ adjacent in ${\cal T}$ are also
adjacent in ${\cal G}$.

Let $\{{\cal T}_j\}_{j\in \N}$ be a family of subtrees in ${\cal
T}$ such that ${\bf V}({\cal T}_j)\cap {\bf V}({\cal T}_{j'})
=\varnothing$ for $j\ne j'$ and $\cup _{j\in \N} {\bf V}({\cal
T}_j) ={\bf V}({\cal T})$. Then $\{{\cal T}_j\} _{j\in \N}$ is
called a partition of the tree ${\cal T}$. Let $\xi_j$ be the
minimal vertex of ${\cal T}_j$. We say that the tree ${\cal T}_s$
succeeds the tree ${\cal T}_j$ (or ${\cal T}_j$ precedes the tree
${\cal T}_s$) if $\xi_j<\xi_s$ and $$\{\xi\in {\cal T}:\; \xi_j\le
\xi<\xi_s\} \subset {\bf V}({\cal T}_j).$$

We consider the function spaces on sets with tree-like structure
from \cite{vas_width_raspr}.

Let $(\Omega, \, \Sigma, \, {\rm mes})$ be a measure space, let
$\hat\Theta$ be a countable partition of $\Omega$ into measurable
subsets, let ${\cal A}$ be a tree with a root such that
\begin{align}
\label{c_v1_a} \exists c_1\ge 1:\quad {\rm card}\, {\bf
V}_1(\xi)\le c_1, \quad \xi \in {\bf V}({\cal A}),
\end{align}
and let $\hat F:{\bf V}({\cal A}) \rightarrow \hat\Theta$ be a
bijective mapping.

Throughout we consider at most countable partitions into
measurable subsets.

Let $1< p\le \infty$, $1\le q< \infty$ be arbitrary numbers. We
suppose that, for any measurable subset $E\subset \Omega$, the
following spaces are defined:
\begin{itemize}
\item the space $X_p(E)$ with seminorm $\|\cdot\|_{X_p(E)}$,
\item the space $Y_q(E)$ with seminorm $\|\cdot\|_{Y_q(E)}$,
\end{itemize}
which all satisfy the following conditions:
\begin{enumerate}
\item $X_p(\Omega)\subset Y_q(\Omega)$;
\item $X_p(E)=\{f|_E:\; f\in X_p(\Omega)\}$, $Y_q(E)=\{f|_E:\; f\in
Y_q(\Omega)\}$;
\item if ${\rm mes}\, E=0$, then $\dim \, Y_q(E)=\dim \, X_p(E)=0$;
\item if $E\subset \Omega$, $E_j\subset \Omega$ ($j\in \N$)
are measurable subsets, $E=\sqcup _{j\in \N} E_j$, then
\begin{align}
\label{f_xp} \|f\|_{X_p(E)}=\left\| \bigl\{
\|f|_{E_j}\|_{X_p(E_j)}\bigr\}_{j\in \N}\right\|_{l_p},\quad f\in
X_p(E),
\end{align}
\begin{align}
\label{f_yq} \|f\|_{Y_q(E)}=\left\| \bigl\{\|f|_{E_j}\|
_{Y_q(E_j)}\bigr\}_{j\in \N}\right\|_{l_q}, \quad f\in Y_q(E);
\end{align}
\item if $E\in \Sigma$, $f\in Y_q(\Omega)$, then $f\cdot \chi_E\in
Y_q(\Omega)$.
\end{enumerate}

Let ${\cal P}(\Omega)\subset X_p(\Omega)$ be a subspace of finite
dimension $r_0$ and let $\|f\|_{X_p(\Omega)}=0$ for any $f\in
{\cal P}(\Omega)$. For each measurable subset $E\subset \Omega$ we
write ${\cal P}(E)=\{P|_E:\; P\in {\cal P}(\Omega)\}$. Let
$G\subset \Omega$ be a measurable subset and let $T$ be a
partition of $G$. We set
\begin{align}
\label{st_omega} {\cal S}_{T}(\Omega)=\{f:\Omega\rightarrow \R:\,
f|_E\in {\cal P}(E), \; f|_{\Omega\backslash G}=0\}.
\end{align}
If $T$ is finite, then ${\cal S}_{T}(\Omega)\subset Y_q(\Omega)$
(see property 5).

For any finite partition $T=\{E_j\}_{j=1}^n$ of the set $E$ and
for each function $f\in Y_q(\Omega)$ we put
$$
\|f\|_{p,q,T}=\left(\sum \limits _{j=1}^n \|f|_{E_j}\|_{Y_q(E_j)}
^{\sigma_{p,q}}\right)^{\frac{1}{\sigma_{p,q}}},
$$
where $\sigma_{p,q}=\min\{p, \, q\}$. Denote by $Y_{p,q,T}(E)$ the
space $Y_q(E)$ with the norm $\|\cdot\|_{p,q,T}$. Notice that
$\|\cdot\| _{Y_q(E)}\le \|\cdot\|_{p,q,T}$.

For each subtree ${\cal A}'\subset {\cal A}$ we set $\Omega
_{{\cal A}'}=\cup _{\xi\in {\bf V}({\cal A}')} \hat F(\xi)$.

\begin{Sup}
\label{sup1} There is a function $w_*:{\bf V}({\cal A})\rightarrow
(0, \, \infty)$ with the following property: for any $\hat\xi\in
{\bf V}({\cal A})$ there exists a linear continuous operator
$P_{\Omega_{{\cal A}_{\hat\xi}}}:Y_q(\Omega)\rightarrow {\cal
P}(\Omega)$ such that for any function $f\in X_p(\Omega)$ and any
subtree ${\cal A}'\subset {\cal A}$ rooted at $\hat\xi$
\begin{align}
\label{f_pom_f} \|f-P_{\Omega_{{\cal
A}_{\hat\xi}}}f\|_{Y_q(\Omega_{{\cal A}'})}\le w_*(\hat\xi)\|f\|
_{X_p(\Omega_{{\cal A}'})}.
\end{align}
\end{Sup}

\begin{Sup}
\label{sup2} There exist a function $\tilde w_*:{\bf V}({\cal
A})\rightarrow (0, \, \infty)$ and numbers $\delta_*>0$, $c_2\ge
1$ such that for each vertex $\xi\in {\bf V}({\cal A})$ and for
any $n\in \N$, $m\in \Z_+$ there is a partition $T_{m,n}(G)$ of
the set $G=\hat F(\xi)$ with the following properties:
\begin{enumerate}
\item ${\rm card}\, T_{m,n}(G)\le c_2\cdot 2^mn$.
\item For any $E\in T_{m,n}(G)$ there exists a linear continuous operator
$P_E:Y_q(\Omega)\rightarrow {\cal P}(E)$ such that for any
function $f\in X_p(\Omega)$
\begin{align}
\label{fpef} \|f-P_Ef\|_{Y_q(E)}\le (2^mn)^{-\delta_*}\tilde
w_*(\xi) \|f\| _{X_p(E)}.
\end{align}
\item For any $E\in T_{m,n}(G)$
\begin{align}
\label{ceptm} {\rm card}\,\{E'\in T_{m\pm 1,n}(G):\, {\rm
mes}(E\cap E') >0\} \le c_2.
\end{align}
\end{enumerate}
\end{Sup}

\begin{Sup}
\label{sup3} There exist $k_*\in \N$, $\lambda_*\ge 0$,
\begin{align}
\label{mu_ge_lambda} \mu_*\ge \lambda_*,
\end{align}
$\gamma_*>0$, absolutely continuous functions $u_*:(0, \, \infty)
\rightarrow (0, \, \infty)$ and $\psi_*:(0, \, \infty) \rightarrow
(0, \, \infty)$, $c_3\ge 1$, $t_0\in \N$, a partition $\{{\cal
A}_{t,i}\}_{t\ge t_0, \, i\in \hat J_t}$ of the tree ${\cal A}$
such that $\lim \limits _{y\to \infty} \frac{yu_*'(y)}{u_*(y)}=0$,
$\lim \limits _{y\to \infty} \frac{y\psi_*'(y)}{\psi_*(y)}=0$,
\begin{align}
\label{w_s_2} c_3^{-1} 2^{-\lambda_*k_*t}u_*(2^{k_*t}) \le
w_*(\xi)\le c_3\cdot 2^{-\lambda_*k_*t}u_*(2^{k_*t}), \quad \xi
\in {\bf V}({\cal A}_{t,i}),
\end{align}
\begin{align}
\label{til_w_s_2} c_3^{-1} 2^{-\mu_*k_*t}u_*(2^{k_*t}) \le \tilde
w_*(\xi)\le c_3\cdot 2^{-\mu_*k_*t}u_*(2^{k_*t}), \quad \xi \in
{\bf V}({\cal A}_{t,i}),
\end{align}
and for $\nu_t:= \sum \limits _{i\in \hat J_t} {\rm card}\, {\bf
V}({\cal A}_{t,i})$ one of the following estimates holds:
\begin{align}
\label{nu_t_k} \nu_t\le c_3\cdot 2^{\gamma_*k_*t}
\psi_*(2^{k_*t})=: c_3 \overline{\nu}_t,\quad t\ge t_0,
\end{align}
or
\begin{align}
\label{nu_t_k1} k_*=1, \quad \nu_t\le c_3\cdot 2^{\gamma_*2^{t}}
\psi_*(2^{2^{t}})=: c_3 \overline{\nu}_t,\quad t\ge t_0.
\end{align}
In addition, we assume that the following assertions hold.
\begin{enumerate}
\item If $p>q$, then
\begin{align}
\label{bipf4684gn} 2^{-\lambda_*k_*t} ({\rm card}\, \hat
J_t)^{\frac 1q-\frac 1p} \le c_3\cdot
2^{-\mu_*k_*t}\overline{\nu}_t^{\frac 1q-\frac 1p}.
\end{align}
\item Let $t$, $t'\in \Z_+$. Then
\begin{align}
\label{2l} 2^{-\lambda_*k_*t'}u_*(2^{k_*t'})\le c_3\cdot
2^{-\lambda_*k_*t}u_*(2^{k_*t}) \quad\text{if}\quad  t'\ge t,
\end{align}
\begin{align}
\label{2ll}
\begin{array}{c} 2^{-\mu_*k_*t'}u_*(2^{k_*t'})\overline{\nu}_{t'}^{\frac 1q-\frac 1p}\le \\ \le c_3\cdot
2^{-\mu_*k_*t}u_*(2^{k_*t})\overline{\nu}_t^{\frac 1q-\frac 1p}
\quad\text{if}\quad t'\ge t, \quad p>q.
\end{array}
\end{align}

\item If the tree ${\cal A}_{t',i'}$ succeeds the tree ${\cal
A}_{t,i}$, then $t'=t+1$.
\end{enumerate}
\end{Sup}
\begin{Rem}
\label{rem_suc} If $\xi \in {\bf V}({\cal A}_{t,i})$, $\xi' \in
{\bf V}({\cal A}_{t',i'})$, $\xi'>\xi$, then $t'>t$.
\end{Rem}

\begin{Rem}
\label{1511} If $p>q$, then from (\ref{2ll}) it follows that
(\ref{nu_t_k1}) cannot hold.
\end{Rem}

\label{ati_label}We introduce some more notation.
\begin{itemize}
\item $\hat \xi_{t,i}$ is the minimal vertex of the tree ${\cal
A}_{t,i}$.
\item $\Gamma _t$ is the maximal subgraph in
${\cal A}$ on the set of vertices $\cup _{i\in \hat J_t} {\bf
V}({\cal A}_{t,i})$, $t\ge t_0$; for $1\le t<t_0$ we put
$\Gamma_t=\varnothing$, $\hat J_t=\varnothing$.
\item $G_t=\cup_{\xi\in {\bf V}(\Gamma_t)}\hat F(\xi)=\cup_{i\in \hat
J_t} \Omega_{{\cal A}_{t,i}}$.
\item $\tilde \Gamma_t$ is the maximal subgraph
on the set of vertices $\cup _{j\ge t} {\bf V}(\Gamma_j)$, $t\in
\N$.
\item $\{\tilde {\cal A}_{t,i}\}_{i\in \overline{J}_t}$
is the set of connected components of the graph $\tilde \Gamma
_t$.
\item $\tilde U_{t,i}=\cup _{\xi\in {\bf V}(\tilde{\cal
A}_{t,i})}\hat F(\xi)$.
\item $\tilde U_t =\cup _{i\in \overline{J}_t} \tilde U_{t,i} =\cup
_{\xi\in {\bf V}(\tilde \Gamma_t)} \hat F(\xi)$.
\end{itemize}
If $t\ge t_0$, then
\begin{align}
\label{v_min_gt} {\bf V}_{\min}(\tilde \Gamma_t)= {\bf
V}_{\min}(\Gamma_t)=\{\hat \xi_{t,i}\}_{i\in \hat J_t}
\end{align}
(see \cite[p. 30]{vas_width_raspr}), and we may assume that
\begin{align}
\label{ovrl_it_eq_hat_it} \overline{J}_t=\hat J_t, \quad t\ge t_0.
\end{align}
The set $\hat J_{t_0}$ is a singleton. Denote $\{i_0\}=\hat
J_{t_0}$.

We set $\mathfrak{Z}_0=(p, \, q, \, r_0,\, w_*, \, \tilde w_*, \,
\delta_*, \, k_*, \, \lambda_*, \, \mu_*,\, \gamma_*, \, \psi_*,\,
u_*,\, c_1, \, c_2, \, c_3)$.

From Assumption \ref{sup1} it follows that there exists a linear
continuous operator $\hat P:Y_q(\Omega) \rightarrow {\cal
P}(\Omega)$ such that for any function $f\in X_p(\Omega)$
\begin{align}
\label{fpfyq} \|f-\hat Pf\|_{Y_q(\Omega)}
\underset{\mathfrak{Z}_0}{\lesssim} \|f\|_{X_p(\Omega)}
\end{align}
(we take as $\hat P$ the operator $P_{\Omega_{{\cal A}_{\hat
\xi_{t_0,i_0}}}}$). We set
$$
\hat X_p(\Omega)=\{f-\hat Pf:\; f\in X_p(\Omega)\}.
$$
Then $\hat X_p(\Omega) \subset X_p(\Omega)$. Moreover, since
$Y_q(\Omega)$ is a normed space and $\|f\|_{X_p(\Omega)}=\|f-\hat
Pf\|_{X_p(\Omega)}$ (by the property of ${\cal P}(\Omega)$), then
$(\hat X_p(\Omega), \, \|\cdot\|_{X_p(\Omega)})$ is a normed
space. Indeed, if $\|f-\hat Pf\|_{X_p(\Omega)}=0$, then $\|f-\hat
Pf\|_{Y_q(\Omega)}\stackrel{(\ref{fpfyq})}{=}0$ and $f-\hat Pf=0$.

Denote by $B\hat X_p(\Omega)$ the unit ball of $\hat X_p(\Omega)$.
Let $I:\hat X_p(\Omega) \rightarrow Y_q(\Omega)$ be the embedding
operator.

The following assertions were proved in \cite[p. 30,
32]{vas_width_raspr}.
\begin{Lem}
\label{obr} There exists $x_0\in (0, \, \infty)$ such that for any
$x\ge x_0$ the equation $y^{\gamma_*}\psi_*(y)=x$ has a unique
solution $y(x)$. Moreover, $y(x)=x^{\beta_*}\varphi_*(x)$, where
$\beta_*=\frac{1}{\gamma_*}$ and $\varphi_*$ is an absolutely
continuous function such that $\lim _{x\to +\infty}
\frac{x\varphi_*'(x)}{\varphi_*(x)}=0$.
\end{Lem}

\begin{Lem}
\label{log} Let $\gamma_*>0$, $\psi_*(y)=|\log
y|^{\alpha_*}\rho_*(|\log y|)$, where $\rho_*:(0, \, \infty)
\rightarrow (0, \, \infty)$ is an absolutely continuous function
such that $\lim \limits _{y\to \infty} \frac{y
\rho_*'(y)}{\rho_*(y)}=0$. Let $\varphi_*$ be such as in Lemma
\ref{obr}. Then for sufficiently large $x>1$
$$
\varphi_*(x)\underset{\gamma_*, \alpha_*,\rho_*}{\asymp} (\log
x)^{-\frac{\alpha_*}{\gamma_*}}\left[\rho_*(\log
x)\right]^{-\frac{1}{\gamma_*}}.
$$
\end{Lem}

\begin{Trm}
\label{trm1} Let $1< p\le \infty$, $1\le q< \infty$, let
assumptions \ref{sup1}, \ref{sup2} and \ref{sup3} with
(\ref{nu_t_k}) hold, and let
\begin{align}
\label{delpq} \delta_*>\left(\frac 1q-\frac 1p\right)_+.
\end{align}
Suppose that $\delta_*\ne \lambda_*\beta_*$ for $p\le q$ and
$\delta_*\ne \mu_*\beta_*$ for $p>q$. Then there exists
$n_0=n_0(\mathfrak{Z}_0)$ such that for any $n\ge n_0$ the
following estimates hold.
\begin{itemize}
\item Let $p\le q$. We set
$$
\sigma_*(n)=\left\{ \begin{array}{l} 1 \quad \text{for}\quad
\delta_*<\lambda_*\beta_*, \\ u_*(n^{\beta_*} \varphi_*(n))
\varphi_*^{-\lambda_*}(n) \quad \text{for}\quad \delta_* >
\lambda_*\beta_*.\end{array}\right.
$$
Then
\begin{align}
\label{vrth_n_pleq} e_n(I:\hat X_p(\Omega)\rightarrow Y_q(\Omega))
\underset{\mathfrak{Z}_0} {\lesssim} n^{-\min(\delta_*, \,
\lambda_*\beta_*)+\frac 1q-\frac 1p}\sigma_*(n).
\end{align}
\item Let $p>q$. We set
$$
\sigma_*(n)=\left\{ \begin{array}{l} 1 \quad \text{for}\quad
\delta_*<\mu_*\beta_*, \\ u_*(n^{\beta_*} \varphi_*(n))
\varphi_*^{-\mu_*}(n) \quad \text{for}\quad \delta_* >
\mu_*\beta_*.\end{array}\right.
$$
Then
\begin{align}
\label{vrth_n_pgq} e_n(I:\hat X_p(\Omega)\rightarrow Y_q(\Omega))
\underset{\mathfrak{Z}_0} {\lesssim} n^{-\min(\delta_*, \,
\mu_*\beta_*)+\frac 1q-\frac 1p}\sigma_*(n).
\end{align}
\end{itemize}
\end{Trm}

\begin{Trm}
\label{trm2} Let $1< p< q< \infty$, let assumptions \ref{sup1},
\ref{sup2} and \ref{sup3} with (\ref{nu_t_k1}) hold, and let
$\delta_*>\left(\frac 1q-\frac 1p\right)_+$. Then
\begin{align}
\label{en2_lgpq} e_n(I:\hat X_p(\Omega)\rightarrow Y_q(\Omega))
\underset{\mathfrak{Z}_0} {\lesssim} n^{\frac 1q-\frac 1p} (\log
n)^{-\lambda_*-\frac 1q+\frac 1p} u_*(\log n)
\end{align}
if $\lambda_*>\frac 1p-\frac 1q$, and
\begin{align}
\label{en2_llpq} e_n(I:\hat X_p(\Omega)\rightarrow Y_q(\Omega))
\underset{\mathfrak{Z}_0} {\lesssim} n^{-\lambda_*} u_*(n)
\end{align}
if $\lambda_*<\frac 1p-\frac 1q$.
\end{Trm}

First we prove some auxiliary assertions.

{\bf Definition of operators $Q_t$.} By assumption \ref{sup1} and
(\ref{w_s_2}), for any $t\ge t_0$, $i\in \hat
J_t\stackrel{(\ref{ovrl_it_eq_hat_it})}{=}\overline{J}_t$ there
exists a linear continuous operator $\tilde P_{t,i}:
Y_q(\Omega)\rightarrow {\cal P}(\Omega)$ such that for any
function $f\in \hat X_p(\Omega)$ and for any subtree ${\cal
A}'\subset {\cal A}$ rooted at $\hat \xi_{t,i}$
\begin{align}
\label{ftp_ti_mod} \|f-\tilde P_{t,i}f\|_{Y_q(\Omega_{{\cal A}'})}
\underset{\mathfrak{Z}_0}{\lesssim}
2^{-\lambda_*k_*t}u_*(2^{k_*t})\|f\| _{X_p(\Omega_{{\cal A}'})};
\end{align}
in particular,
\begin{align}
\label{ftp_ti} \|f-\tilde P_{t,i}f\|_{Y_q(\tilde U_{t,i})}
\underset{\mathfrak{Z}_0}{\lesssim}
2^{-\lambda_*k_*t}u_*(2^{k_*t})\|f\| _{X_p(\tilde U_{t,i})}.
\end{align}
Moreover, from the definitions of the space $\hat X_p(\Omega)$ and
of the operator $\hat P$ it follows that we can set $\tilde
P_{t_0,i_0}=0$. Let $1\le t< t_0$. Then $\overline{J}_t=\{i_0\}$,
$\tilde U_{t,i_0}=\tilde U_{t_0,i_0}$. We set $\tilde
P_{t,i_0}=0$. By (\ref{2l}) we get that (\ref{ftp_ti}) holds.

Let
\begin{align}
\label{qtf_x} Q_tf(x)=\tilde P_{t,i}f(x) \quad \text{for} \quad
x\in \tilde U_{t,i}, \quad i\in \overline{J}_t, \quad
Q_tf(x)=0\quad \text{for} \quad x\in \Omega \backslash \tilde U_t,
\end{align}
\begin{align}
\label{tt_def} T_t=\{\tilde U_{t+1,i}\}_{i\in \overline{J}_{t+1}}.
\end{align}
Then $(Q_{t+1}f-Q_tf)\chi_{\tilde U_{t+1}}\in {\cal
S}_{T_t}(\Omega)$, and for $p\le q$ we have
\begin{align}
\label{pp3_11} \|f-Q_tf\| _{Y_q(\tilde U_t)}\le \|f-Q_tf\|
_{Y_{p,q,T_{t-1}}(\tilde U_t)}
\stackrel{(\ref{ftp_ti})}{\underset{\mathfrak{Z}_0} {\lesssim}}
2^{-\lambda_*k_*t}u_*(2^{k_*t}).
\end{align}
Notice that if $t<t_0$, then $Q_tf=Q_{t+1}f=0$ (since $\tilde
P_{t,i_0}=0$ for $t\le t_0$ by definition).

{\bf Definition of operators $P_{t,m}$.} For $t\ge t_0$ we set
\begin{align}
\label{mt_def} m_t=\lceil \log \nu_t\rceil.
\end{align}
In \cite{vas_width_raspr} for each $m\in \Z_+$ the set $G_{m,t}
\subset G_t$, the partition $\tilde T_{t,m}$ of $G_{m,t}$ and the
linear continuous operator $P_{t,m}:Y_q(\Omega) \rightarrow {\cal
S}_{\tilde T_{t,m}}(\Omega)$ were constructed with the following
properties:
\begin{enumerate}
\item $G_{m,t}\subset G_{m+1,t}$, $G_{m_t,t}=G_t$;
\item for any $m\in \Z_+$
\begin{align}
\label{cttm} {\rm card}\, \tilde T_{t,m}
\underset{\mathfrak{Z}_0}{\lesssim} 2^m;
\end{align}
\item for any function $f\in \hat X_p(\Omega)$ and for any set
$E\in \tilde T_{t,m}$
\begin{align}
\label{fpttm1} \|f-P_{t,m}f\|_{Y_q(E)}
\underset{\mathfrak{Z}_0}{\lesssim} 2^{-\lambda_*k_*t}
u_*(2^{k_*t})\|f\|_{X_p(E)}, \quad m\le m_t,
\end{align}
\begin{align}
\label{fpttm2} \|f-P_{t,m}f\|_{Y_q(E)}
\underset{\mathfrak{Z}_0}{\lesssim} 2^{-\mu_*k_*t}
u_*(2^{k_*t})\cdot 2^{-\delta_*(m-m_t)}\|f\|_{X_p(E)}, \quad m>
m_t;
\end{align}
\item for any set $E\in \tilde T_{t,m}$
\begin{align}
\label{card_et} {\rm card}\, \{E'\in \tilde T_{t,m\pm 1}:\ {\rm
mes} (E\cap E')>0\} \underset{\mathfrak{Z}_0}{\lesssim} 1.
\end{align}
\end{enumerate}
Moreover, we may assume that $\tilde T_{m_t,t}=\{\hat
F(\xi)\}_{\xi \in {\bf V}(\Gamma_t)}$.

\begin{Rem}
In \cite{vas_width_raspr} the relation (\ref{fpttm2}) was proved
with $\lambda_*$ instead of $\mu_*$; however it follows from the
construction and from assumption \ref{sup2} that the estimate with
$\mu_*$ is also correct (see \cite[proof of formula
(80)]{vas_width_raspr}).
\end{Rem}

\smallskip

Let
\begin{align}
\label{t_st_n} t_*(n)=\min \{t\in \N:\; \overline{\nu}_t\ge n\},
\end{align}
\begin{align}
\label{t_st_st_n} t_{**}(n)=\left\{ \begin{array}{l} t_*(n), \quad
\text{if} \quad p\ge q, \\ \min \{t\in \N:\; \overline{\nu}_t\ge
2^n\}, \quad \text{if} \quad p<q.\end{array}\right.
\end{align}

The following lemma is proved in \cite[formula (60)]{vas_bes}.
\begin{Lem}
\label{sum_lem} Let $\Lambda_*:(0, \, +\infty) \rightarrow (0, \,
+\infty)$ be an absolutely continuous function such that $\lim
\limits_{y\to +\infty}\frac{y\Lambda _*'(y)} {\Lambda _*(y)}=0$.
Then for any  $\varepsilon >0$
\begin{align}
\label{sum_l_est} t^{-\varepsilon}
\underset{\varepsilon,\Lambda_*}{\lesssim}
\frac{\Lambda_*(ty)}{\Lambda_*(y)}\underset{\varepsilon,
\Lambda_*}{\lesssim} t^\varepsilon,\quad 1\le y<\infty, \;\; 1\le
t<\infty.
\end{align}
\end{Lem}

\begin{Sta}
\label{nu_t_st} If (\ref{nu_t_k}) holds, then
\begin{align}
\label{nu_t_est1} \overline{\nu}_{t_*(n)}
\underset{\mathfrak{Z}}{\asymp} n, \quad 2^{k_*t_*(n)}
\underset{\mathfrak{Z}}{\asymp} n^{\beta_*} \varphi_*(n),
\end{align}
and in the case $p<q$ we have
\begin{align}
\label{nu_t_est1_1} \overline{\nu}_{t_{**}(n)}
\underset{\mathfrak{Z}}{\asymp} 2^n, \quad 2^{k_*t_{**}(n)}
\underset{\mathfrak{Z}}{\asymp} 2^{\beta_*n} \varphi_*(2^n).
\end{align}
If (\ref{nu_t_k1}) holds, then for sufficiently large $n\in \N$
\begin{align}
\label{nu_t_est2} 2^{t_*(n)} \underset{\mathfrak{Z}}{\asymp} \log
n; \quad \text{if} \quad p<q, \quad \text{then} \quad
2^{t_{**}(n)} \underset{\mathfrak{Z}}{\asymp} n.
\end{align}
\end{Sta}
\begin{proof}
Estimates (\ref{nu_t_est1}) follow from \cite[formula
(51)]{vas_width_raspr}. The relations (\ref{nu_t_est1_1}) can be
proved similarly.

Let us prove the first estimate in (\ref{nu_t_est2}); the second
relation is proved similarly. We have
$$
n\stackrel{(\ref{nu_t_k1}), (\ref{t_st_n})}{\le} 2^{\gamma_*\cdot
2^{t_*(n)}} \psi_*(2^{2^{t_*(n)}})
\stackrel{(\ref{sum_l_est})}{\underset{\mathfrak{Z}_0}{\lesssim}}
2^{2\gamma_*\cdot 2^{t_*(n)}};
$$
therefore, $2^{t_*(n)} \underset{\mathfrak{Z}_0}{\gtrsim} \log n$
for sufficiently large $n\in \N$. Further,
$$
n \stackrel{(\ref{nu_t_k1}), (\ref{t_st_n})}{>} 2^{\gamma_*\cdot
2^{t_*(n)-1}} \psi_*(2^{2^{t_*(n)-1}})
\stackrel{(\ref{sum_l_est})}{\underset{\mathfrak{Z}_0}{\gtrsim}}
2^{\gamma_*\cdot 2^{t_*(n)-2}},
$$
and $2^{t_*(n)} \underset{\mathfrak{Z}_0}{\lesssim} \log n$ for
sufficiently large $n\in \N$.
\end{proof}

It is proved in \cite[p. 37--39]{vas_width_raspr} that for $f\in
\hat X_p(\Omega)$
\begin{align}
\label{f_q_t} \begin{array}{c} f= \sum \limits _{j=t_0}
^{t_{**}(n)-1} (Q_{j+1}f-Q_jf)\chi _{\tilde U_{j+1}} +\sum \limits
_{j=t_0} ^{t_{**}(n)-1} (f-Q_jf) \chi _{G_j}
+(f-Q_{t_{**}(n)}f)\chi _{\tilde U_{t_{**}(n)}},
\end{array}
\end{align}
\begin{align}
\label{f_p_t_m} \begin{array}{c} \sum \limits
_{t=t_0}^{t_{**}(n)-1} (f-Q_tf)\chi_{G_t} =\sum \limits
_{t=t_0}^{t_*(n)-1} \sum \limits_{m=0}^\infty
(P_{t,m+1}f-P_{t,m}f)\chi _{G_{m,t}} +\\+\sum \limits
_{t=t_*(n)}^{t_{**}(n)-1} \sum \limits _{m=m_t}^\infty
(P_{t,m+1}f-P_{t,m}f)\chi _{G_t} + \sum \limits
_{t=t_*(n)}^{t_{**}(n)-1} (P_{t,m_t}f-Q_tf)\chi _{G_t},
\end{array}
\end{align}
\begin{align}
\label{f_p_t_m1} \begin{array}{c} \sum \limits
_{t=t_0}^{t_{**}(n)-1} (f-Q_tf)\chi_{G_t} =\sum \limits
_{t=t_0}^{t_*(n)-1} \sum \limits_{m=m_t+1}^\infty
(P_{t,m+1}f-P_{t,m}f)\chi _{G_{m,t}} +\\+\sum \limits
_{t=t_0}^{t_{*}(n)-1} (P_{t,m_t+1}f-Q_tf)\chi _{G_t} +\sum \limits
_{t=t_*(n)}^{t_{**}(n)-1} \sum \limits _{m=m_t}^\infty
(P_{t,m+1}f-P_{t,m}f)\chi _{G_t} + \\ +\sum \limits
_{t=t_*(n)}^{t_{**}(n)-1} (P_{t,m_t}f-Q_tf)\chi _{G_t}.
\end{array}
\end{align}

\begin{Lem}
\label{qt_st_n}
\begin{enumerate}
\item Let (\ref{nu_t_k}) hold. Then for any $f\in B\hat X_p(\Omega)$
\begin{align}
\label{est_case1_pgq} \|f-Q_{t_{**}(n)}f\|_{Y_q(\tilde
U_{t_{**}(n)})} \underset{\mathfrak{Z}_0}{\lesssim}
n^{-\mu_*\beta_*+\frac 1q-\frac 1p} \varphi_*^{-\mu_*}(n)
u_*(n^{\beta_*}\varphi_*(n)) \quad \text{if} \quad p>q,
\end{align}
\begin{align}
\label{est_case1_peq} \|f-Q_{t_{**}(n)}f\|_{Y_q(\tilde
U_{t_{**}(n)})} \underset{\mathfrak{Z}_0}{\lesssim}
n^{-\lambda_*\beta_*} \varphi_*^{-\lambda_*}(n)
u_*(n^{\beta_*}\varphi_*(n)) \quad \text{if} \quad p=q,
\end{align}
\begin{align}
\label{est_case1_plq} \|f-Q_{t_{**}(n)}f\|_{Y_q(\tilde
U_{t_{**}(n)})} \underset{\mathfrak{Z}_0}{\lesssim}
2^{-\lambda_*\beta_*n} \varphi_*^{-\lambda_*}(2^{n})
u_*(2^{\beta_*n}\varphi_*(2^{n})) \quad \text{if} \quad p<q.
\end{align}
\item Let (\ref{nu_t_k1}) hold (by Remark \ref{1511}, this case is possible only for $p\le q$).
Then for $f\in B\hat X_p(\Omega)$
\begin{align}
\label{est_case2_peq} \|f-Q_{t_{**}(n)}f\|_{Y_q(\tilde
U_{t_{**}(n)})} \underset{\mathfrak{Z}_0}{\lesssim} (\log
n)^{-\lambda_*}u_*(\log n)\quad \text{if} \quad p=q,
\end{align}
\begin{align}
\label{est_case2_plq} \|f-Q_{t_{**}(n)}f\|_{Y_q(\tilde
U_{t_{**}(n)})} \underset{\mathfrak{Z}_0}{\lesssim} n^{-\lambda_*}
u_*(n)\quad \text{if} \quad p<q.
\end{align}
\end{enumerate}
\end{Lem}
\begin{proof}
Estimates from assertion 1 follow from (\ref{pp3_11}),
(\ref{sum_l_est}), (\ref{nu_t_est1}), (\ref{nu_t_est1_1}).
Estimates from assertion 2 follow from (\ref{pp3_11}),
(\ref{sum_l_est}) and (\ref{nu_t_est2}).
\end{proof}

Recall the notation $\sigma_{p,q}=\min\{p, \, q\}$.
\begin{Lem}
\label{oper_a} {\rm (see \cite{vas_width_raspr}).} Let $T$ be a
finite partition of a measurable subset $G\subset \Omega$,
$\nu=\dim {\cal S}_T(\Omega)$ (see (\ref{st_omega})). Then there
exists a linear isomorphism $A:{\cal S}_T(\Omega)\rightarrow
\R^\nu$ such that $\|A\|_{Y_{p,q,T}(G)\rightarrow
l_{\sigma_{p,q}}^\nu}\underset{\sigma_{p,q}, \, r_0}{\lesssim} 1$,
$\|A^{-1}\| _{l_q^\nu\rightarrow Y_q(G)} \underset{q, \,
r_0}{\lesssim} 1$.
\end{Lem}

\begin{Lem}
\label{sum_qt_est} There exists a sequence
$\{k_t\}_{t=t_0}^{t_{**}(n)-1}\subset \N$ such that
\begin{align}
\label{slttkt} \sum \limits _{t=t_0}^{t_{**}(n)-1}(k_t-1)
\underset{\mathfrak{Z}_0}{\lesssim} n
\end{align}
and the following assertions hold.
\begin{enumerate}
\item Suppose that (\ref{nu_t_k}) holds. Then for $p>q$
\begin{align}
\label{sum_ekt_1_pgq} \sum \limits _{t=t_0}^{t_{**}(n)-1}
e_{k_t}(Q_{t+1}-Q_t:\hat X_p(\Omega) \rightarrow Y_q(\tilde
U_{t+1}))\underset{\mathfrak{Z}_0}{\lesssim} n^{-\mu_*\beta_*+
\frac 1q-\frac 1p} \varphi_*^{-\mu_*}(n)
u_*(n^{\beta_*}\varphi_*(n)),
\end{align}
and for $p\le q$
\begin{align}
\label{sum_ekt_1_pleq} \sum \limits _{t=t_0}^{t_{**}(n)-1}
e_{k_t}(Q_{t+1}-Q_t:\hat X_p(\Omega) \rightarrow Y_q(\tilde
U_{t+1}))\underset{\mathfrak{Z}_0}{\lesssim} n^{-\lambda_*\beta_*+
\frac 1q-\frac 1p} \varphi_*^{-\lambda_*}(n)
u_*(n^{\beta_*}\varphi_*(n)).
\end{align}
\item Suppose that (\ref{nu_t_k1}) holds (by Remark \ref{1511}, this case is possible only for $p\le
q$). If $p=q$, then
\begin{align}
\label{sum_ekt_2_peq} \sum \limits _{t=t_0}^{t_{**}(n)-1}
e_{k_t}(Q_{t+1}-Q_t:\hat X_p(\Omega) \rightarrow Y_q(\tilde
U_{t+1}))\underset{\mathfrak{Z}_0}{\lesssim} (\log n)^{-\lambda_*}
u_*(\log n);
\end{align}
if $p<q$ and $\lambda_*>\frac 1p-\frac 1q$, then
\begin{align}
\label{sum_ekt_2_plq1} \sum \limits _{t=t_0}^{t_{**}(n)-1}
e_{k_t}(Q_{t+1}-Q_t:\hat X_p(\Omega) \rightarrow Y_q(\tilde
U_{t+1}))\underset{\mathfrak{Z}_0}{\lesssim} n^{\frac 1q-\frac
1p}(\log n)^{-\lambda_*+\frac 1p-\frac 1q} u_*(\log n);
\end{align}
if $p<q$ and $\lambda_*<\frac 1p-\frac 1q$, then
\begin{align}
\label{sum_ekt_2_plq11} \sum \limits _{t=t_0}^{t_{**}(n)-1}
e_{k_t}(Q_{t+1}-Q_t:\hat X_p(\Omega) \rightarrow Y_q(\tilde
U_{t+1}))\underset{\mathfrak{Z}_0}{\lesssim} n^{-\lambda_*}u_*(n).
\end{align}
\end{enumerate}
\end{Lem}
\begin{proof}
We set $s_t=\dim {\cal S}_{T_t}(\Omega)$. By (\ref{nu_t_k}),
(\ref{nu_t_k1}) and (\ref{tt_def}), there exists
$C=C(\mathfrak{Z}_0)\ge 1$ such that
\begin{align}
\label{stc} s_t\le C\overline{\nu}_t=:\overline{s}_t.
\end{align}
By Lemma \ref{oper_a}, there exists an operator $A_t:{\cal
S}_{T_t}(\Omega) \rightarrow \R^{s_t}$ such that
\begin{align}
\label{atypqt} \|A_t\|_{Y_{p,q,T_t}\rightarrow
l_{\sigma_{p,q}}^{s_t}} \underset{\sigma_{p,q},r_0}{\lesssim} 1,
\quad \|A_t^{-1}\|_{l_q^{s_t}\rightarrow Y_q(\Omega)}
\underset{q,r_0} {\lesssim} 1.
\end{align}
By (\ref{mult_n}), (\ref{stc}) and (\ref{atypqt}),
$$
\sum \limits _{t=t_0}^{t_{**}(n)-1}e_{k_t}(Q_{t+1}-Q_t:\hat
X_p(\Omega) \rightarrow Y_q(\tilde
U_{t+1}))\underset{\mathfrak{Z}_0}{\lesssim}
$$$$\lesssim\sum \limits _{t=t_0}^{t_{**}(n)-1} \|Q_{t+1}-Q_t\|
_{\hat X_p(\Omega) \rightarrow Y_{p,q,T_t}(\Omega)}
e_{k_t}(I_{\overline{s}_t}:l_{\sigma_{p,q}}^{\overline{s}_t}
\rightarrow l_q^{\overline{s}_t}).
$$

Let $p>q$. Then $\|\cdot\|_{p,q,T_t}=\|\cdot\|_{Y_q(\tilde
U_{t+1})}$. It was proved in \cite[Step 3 in the proof of Theorem
2]{vas_width_raspr} that for $f\in B\hat X_p(\Omega)$
\begin{align}
\label{fqtf} \|f-Q_tf\|_{Y_q(\tilde U_t)}
\underset{\mathfrak{Z}_0}{\lesssim} 2^{-\mu_*k_*t}
u_*(2^{k_*t})\overline{\nu}_t^{\frac 1q-\frac 1p}.
\end{align}
Hence,
\begin{align}
\label{qqt} \|Q_{t+1}-Q_t\| _{\hat X_p(\Omega) \rightarrow
Y_{p,q,T_t}(\Omega)}\underset{\mathfrak{Z}_0}{\lesssim}
2^{-\mu_*k_*t} u_*(2^{k_*t})\overline{\nu}_t^{\frac 1q-\frac 1p},
\quad p>q.
\end{align}

Let $p\le q$. It was proved in \cite[Step 3 in the proof of
Theorem 2]{vas_width_raspr} that for $f\in B\hat X_p(\Omega)$
$$
\|f-Q_tf\|_{p,q,T_t}\underset{\mathfrak{Z}_0}{\lesssim}
2^{-\lambda_*k_*t} u_*(2^{k_*t}),
$$
which implies
\begin{align}
\label{qqt1} \|Q_{t+1}-Q_t\| _{\hat X_p(\Omega) \rightarrow
Y_{p,q,T_t}(\Omega)}\underset{\mathfrak{Z}_0}{\lesssim}
2^{-\lambda_*k_*t} u_*(2^{k_*t}), \quad p\le q.
\end{align}

Let $\varepsilon>0$ (it will be chosen later by $\mathfrak{Z}_0$).
Denote
$$
\hat t(n)=t_*(n), \quad \text{if (\ref{nu_t_k}) holds}, \quad
\text{or if (\ref{nu_t_k1}) holds and } \lambda_*>\frac 1p-\frac
1q,
$$
and $\hat t(n)=t_{**}(n)$, if (\ref{nu_t_k1}) holds and
$\lambda_*<\frac 1p-\frac 1q$.

We set
\begin{align}
\label{kt_def} k_t=\left\{ \begin{array}{l} \lceil n\cdot
2^{-\varepsilon(t_*(n)-t)}\rceil \quad \text{for} \quad t< t_*(n),
\\ \lceil n\cdot 2^{-\varepsilon|t-\hat t(n)|}\rceil \quad
\text{for} \quad t_*(n)\le t< t_{**}(n). \end{array} \right.
\end{align}
Then (\ref{slttkt}) holds.

By (\ref{nu_t_k}), (\ref{nu_t_k1}), (\ref{stc}) and
(\ref{kt_def}), there exists $l_*=l_*(\mathfrak{Z}_0)\in \N$ such
that for sufficiently small $\varepsilon>0$ and for any $0\le
r<l_*$ the sequence $\left\{\frac{k_{l_*t+r}}
{\overline{s}_{l_*t+r}} \right\}_{l_*t+r\le t_*(n)-1}$ decreases
not slower than some geometric progression. Moreover, for $t<
t_*(n)$ we have $\frac{\overline{s}_t}{k_t}
\underset{\mathfrak{Z}_0} {\lesssim} 1$. Therefore, by Theorem
\ref{shutt_trm},
$$e_{k_t}(I_{\overline{s}_t} :l_{\sigma_{p,q}} ^{\overline{s}_t}
\rightarrow l_q^{\overline{s}_t}) \underset{p,q}{\asymp}
2^{-\frac{k_t}{\overline{s}_t}}\overline{s}_t^{\frac 1q-\frac
{1}{\sigma_{p,q}}}, \quad t<t_*(n).$$

Let $p>q$. By Remark \ref{1511}, we have (\ref{nu_t_k}). Hence,
$$
\sum \limits _{t=t_0}^{t_*(n)-1} \|Q_{t+1}-Q_t\| _{\hat
X_p(\Omega) \rightarrow Y_{p,q,T_t}(\Omega)}
e_{k_t}(I_{\overline{s}_t}:l_{\sigma_{p,q}}^{\overline{s}_t}
\rightarrow l_q^{\overline{s}_t})
\stackrel{(\ref{qqt})}{\underset{\mathfrak{Z}_0}{\lesssim}}
$$$$\lesssim 2^{-\mu_*k_*t_*(n)} u_*(2^{k_*t_*(n)})\overline{\nu}
_{t_*(n)}^{\frac 1q-\frac 1p}
\stackrel{(\ref{nu_t_est1})}{\underset{\mathfrak{Z}_0}{\asymp}}
n^{-\mu_*\beta_*+\frac 1q-\frac 1p} \varphi_*^{-\mu_*}(n)
u_*(n^{\beta_*}\varphi_*(n)).
$$
If $p\le q$ and (\ref{nu_t_k}) holds, then
$$
\sum \limits _{t=t_0}^{t_*(n)-1} \|Q_{t+1}-Q_t\| _{\hat
X_p(\Omega) \rightarrow Y_{p,q,T_t}(\Omega)}
e_{k_t}(I_{\overline{s}_t}:l_{\sigma_{p,q}}^{\overline{s}_t}
\rightarrow l_q^{\overline{s}_t})
\stackrel{(\ref{qqt1})}{\underset{\mathfrak{Z}_0}{\lesssim}}
$$$$\lesssim 2^{-\lambda_*k_*t_*(n)} u_*(2^{k_*t_*(n)})
\overline{\nu}_{t_*(n)}^{\frac 1q-\frac 1p}
\stackrel{(\ref{nu_t_est1})}{\underset{\mathfrak{Z}_0}{\asymp}}
n^{-\lambda_*\beta_*+\frac 1q-\frac 1p} \varphi_*^{-\lambda_*}(n)
u_*(n^{\beta_*}\varphi_*(n)).
$$
If $p\le q$ and (\ref{nu_t_k1}) holds, then
$$
\sum \limits _{t=t_0}^{t_*(n)-1} \|Q_{t+1}-Q_t\| _{\hat
X_p(\Omega) \rightarrow Y_{p,q,T_t}(\Omega)}
e_{k_t}(I_{\overline{s}_t}:l_{\sigma_{p,q}}^{\overline{s}_t}
\rightarrow l_q^{\overline{s}_t})
\stackrel{(\ref{qqt1})}{\underset{\mathfrak{Z}_0}{\lesssim}}
$$
$$
\lesssim 2^{-\lambda_*k_*t_*(n)} u_*(2^{k_*t_*(n)}) \cdot \max
_{0\le r\le l_*}
2^{-\frac{k_{t_*(n)-1-r}}{\overline{s}_{t_*(n)-1-r}}}\overline{s}_{t_*(n)-1-r}^{\frac
1q-\frac 1p} \underset{\mathfrak{Z}_0}{\lesssim}
$$
$$
\lesssim 2^{-\lambda_*k_*t_*(n)} u_*(2^{k_*t_*(n)})
k_{t_*(n)-1}^{\frac 1q-\frac 1p}
\stackrel{(\ref{nu_t_est2})}{\underset{\mathfrak{Z}_0}{\lesssim}}
 n^{\frac 1q-\frac 1p} (\log n)^{-\lambda_*}u_*(\log n).
$$

If $p\ge q$, we get the desired estimates in Lemma since
$t_*(n)=t_{**}(n)$.

Let $p<q$ (then $\sigma_{p,q}=p$). We apply Theorem
\ref{shutt_trm}. For $t_*(n)\le t<t_{**}(n)$ we have
$$e_{k_t}(I_{\overline{s}_t}:l_{\sigma_{p,q}}^{\overline{s}_t}
\rightarrow l_q^{\overline{s}_t}) \underset{p,q}{\lesssim}
k_t^{\frac 1q-\frac 1p} \left(\log
\left(1+\frac{\overline{\nu}_t}{k_t}\right)\right)^{\frac 1p-\frac
1q}\underset{\mathfrak{Z}_0}{\lesssim} k_t^{\frac 1q-\frac 1p}
(\log \overline{\nu}_t)^{\frac 1p-\frac 1q}.$$

If (\ref{nu_t_k}) holds, then for sufficiently small
$\varepsilon>0$
$$
\sum \limits _{t=t_*(n)}^{t_{**}(n)-1} \|Q_{t+1}-Q_t\| _{\hat
X_p(\Omega) \rightarrow Y_{p,q,T_t}(\Omega)}
e_{k_t}(I_{\overline{s}_t}:l_{\sigma_{p,q}}^{\overline{s}_t}
\rightarrow l_q^{\overline{s}_t}) \stackrel{(\ref{qqt1}),
(\ref{kt_def})}{\underset{\mathfrak{Z}_0}{\lesssim}}
$$
$$
\lesssim \sum \limits _{t=t_*(n)}^{t_{**}(n)-1} 2^{-\lambda_*k_*t}
u_*(2^{k_*t}) n^{\frac 1q-\frac 1p} 2^{\varepsilon\left(\frac
1p-\frac 1q\right)(t-t_*(n))} \left(\log \left(1+
\frac{\overline{\nu}_t}{k_t}\right)\right)^{\frac 1p-\frac 1q}
\stackrel{(\ref{nu_t_est1}),
(\ref{kt_def})}{\underset{\mathfrak{Z}_0}{\lesssim}}
$$
$$
\lesssim n^{-\lambda_*\beta_* +\frac 1q-\frac 1p}
\varphi_*^{-\lambda_*}(n) u_*(n^{\beta_*}\varphi_*(n)).
$$
If (\ref{nu_t_k1}) holds, then
$$
\sum \limits _{t=t_*(n)}^{t_{**}(n)-1} \|Q_{t+1}-Q_t\| _{\hat
X_p(\Omega) \rightarrow Y_{p,q,T_t}(\Omega)}
e_{k_t}(I_{\overline{s}_t}:l_{\sigma_{p,q}}^{\overline{s}_t}
\rightarrow l_q^{\overline{s}_t})
\stackrel{(\ref{nu_t_k1}),(\ref{sum_l_est}),(\ref{qqt1}),
(\ref{kt_def})}{\underset{\mathfrak{Z}_0}{\lesssim}}
$$
$$
\lesssim \sum \limits _{t=t_*(n)}^{t_{**}(n)-1} 2^{-\lambda_*t}
u_*(2^{t}) n^{\frac 1q-\frac 1p} 2^{\varepsilon\left(\frac
1p-\frac 1q\right)|t-\hat t(n)|} 2^{t\left(\frac 1p-\frac
1q\right)}=:S.
$$
If $\lambda_*>\frac 1p-\frac 1q$, then $S
\stackrel{(\ref{nu_t_est2})}{\underset{\mathfrak{Z}_0}{\lesssim}}
n^{\frac 1q-\frac 1p} (\log n)^{-\lambda_*+\frac 1p-\frac 1q}
u_*(\log n)$. If $\lambda_*<\frac 1p-\frac 1q$, then $S
\stackrel{(\ref{nu_t_est2})}{\underset{\mathfrak{Z}_0}{\lesssim}}
n^{-\lambda_*} u_*(n)$. This completes the proof.
\end{proof}

For $t\ge t_0$, $m\in \Z_+$ we set
\begin{align}
\label{hattm} \hat T_{t,m}=\{E\cap E':\; E\in \tilde T_{t,m}, \;
E'\in \tilde T_{t,m+1}, \;\; {\rm mes}(E\cap E')>0\}.
\end{align}
Let
\begin{align}
\label{stmdef} \tilde s_{t,m}=\dim {\cal S}_{\tilde
T_{t,m}}(\Omega), \quad s_{t,m}=\dim {\cal S}_{\hat
T_{t,m}}(\Omega).
\end{align}
From (\ref{cttm}) and (\ref{card_et}) it follows that there exists
$C_1(\mathfrak{Z}_0)\ge 1$ such that
\begin{align}
\label{tstm} \tilde s_{t,m}\le C_1\cdot 2^m, \quad s_{t,m}\le
C_1\cdot 2^m.
\end{align}
By Lemma \ref{oper_a}, there exists a linear isomorphism
$A_{t,m}:{\cal S}_{\hat T_{t,m}}(\Omega) \rightarrow \R^{s_{t,m}}$
such that
\begin{align}
\label{atm_m1} \|A_{t,m}\|_{Y_{p,q,\hat
T_{t,m}}(G_{m,t})\rightarrow l_{\sigma_{p,q}}^{s_{t,m}}}
\underset{\sigma_{p,q},r_0}{\lesssim} 1, \quad
\|A_{t,m}^{-1}\|_{l_q^{s_{t,m}}\rightarrow Y_q(G_{m,t})}
\underset{q,r_0} {\lesssim} 1.
\end{align}

\begin{Lem}
\label{ptm_sum} There exists a sequence $\{k_{t,m}\}_{t_0\le t<
t_*(n), m\in \Z_+}\subset \N$ such that $\sum \limits _{t_0\le t<
t_*(n), m\in \Z_+}(k_{t,m}-1) \underset{\mathfrak{Z}_0}{\lesssim}
n$ and the following assertions hold.
\begin{enumerate}
\item Let (\ref{nu_t_k}) hold, let the sequence $\{k_t\}_{t\ge t_0}$
be defined by (\ref{kt_def}), and let $\sigma_*(n)$ be such as in
Theorem \ref{trm1}. Then for $p>q$
$$
\sum \limits _{t=t_0}^{t_*(n)-1}\sum \limits _{m=m_t+1}^\infty
e_{k_{t,m}}(P_{t,m+1}-P_{t,m}:\hat X_p(\Omega) \rightarrow
Y_q(G_t))+$$$$+ \sum \limits _{t=t_0}^{t_*(n)-1}
e_{k_t}(P_{t,m_t+1}- Q_t:\hat X_p(\Omega) \rightarrow
Y_q(G_t))\underset{\mathfrak{Z}_0}{\lesssim} n^{-\min\{\delta_*,
\, \mu_*\beta_*\}+\frac 1q-\frac 1p} \sigma_*(n),
$$
and for $p\le q$
$$
\sum \limits _{t=t_0}^{t_*(n)-1}\sum \limits _{m=0}^\infty
e_{k_{t,m}}(P_{t,m+1}-P_{t,m}:\hat X_p(\Omega) \rightarrow
Y_q(G_{m,t})) \underset{\mathfrak{Z}_0}{\lesssim}
n^{-\min\{\delta_*, \, \lambda_*\beta_*\}+\frac 1q-\frac 1p}
\sigma_*(n).
$$
\item Suppose that (\ref{nu_t_k1}) holds. Then
$$
\sum \limits _{t=t_0}^{t_*(n)-1}\sum \limits _{m=0}^\infty
e_{k_{t,m}}(P_{t,m+1}-P_{t,m}:\hat X_p(\Omega) \rightarrow
Y_q(G_{m,t})) \underset{\mathfrak{Z}_0}{\lesssim} n^{\frac
1q-\frac 1p}(\log n)^{-\lambda_*+\frac 1p-\frac 1q}u_*(\log n).
$$
\end{enumerate}
\end{Lem}
\begin{proof}
By (\ref{mult_n}) and (\ref{atm_m1}), it suffices to estimate
\begin{align}
\label{slimtt0} \begin{array}{c} \sum \limits
_{t=t_0}^{t_*(n)-1}\sum \limits _{m=m_t+1}^\infty
\|P_{t,m+1}-P_{t,m}\|_{\hat X_p(\Omega) \rightarrow Y_q(G_t)}
e_{k_{t,m}}(I_{s_{t,m}}: l^{s_{t,m}}_q \rightarrow l_q^{s_{t,m}})+
\\
+\sum \limits _{t=t_0}^{t_*(n)-1} \|P_{t,m_t+1}-Q_t\|_{\hat
X_p(\Omega) \rightarrow Y_q(G_t)} e_{k_{t}}(I_{\tilde
s_{t,m_t+1}}: l^{\tilde s_{t,m_t+1}}_q \rightarrow l_q^{\tilde
s_{t,m_t+1}})=:S_1
\end{array}
\end{align}
for $p>q$ and
\begin{align}
\label{slimtt01} \sum \limits _{t=t_0}^{t_*(n)-1}\sum \limits
_{m=0}^\infty \|P_{t,m+1}-P_{t,m}\|_{\hat X_p(\Omega) \rightarrow
Y_{p,q,\tilde T_{t,m}}(\Omega)} e_{k_{t,m}}(I_{s_{t,m}}:
l^{s_{t,m}}_{p} \rightarrow l_q^{s_{t,m}})=:S_2
\end{align}
for $p\le q$.

Let
\begin{align}
\label{stm} \overline{s}_{t,m}=\lceil C_1\cdot 2^m\rceil.
\end{align}

We define the number $t_1(n)$ as follows. In assertion 1 of Lemma
we set
\begin{align}
\label{t11} t_1(n)=\left\{ \begin{array}{l} t_0 \quad
\text{if}\quad\delta_*<\mu_*\beta_*, \\
t_*(n), \quad \text{if}\quad \delta_*>\mu_*\beta_*
\end{array}\right.
\end{align}
for $p>q$ and
\begin{align}
\label{t12} t_1(n)=\left\{ \begin{array}{l} t_0, \quad
\text{if}\quad\delta_*<\lambda_*\beta_*, \\
t_*(n), \quad \text{if}\quad \delta_*>\lambda_*\beta_*
\end{array}\right.
\end{align}
for $p\le q$. In assertion 2 we set
\begin{align}
\label{t13} t_1(n)=t_*(n).
\end{align}
Denote
\begin{align}
\label{mtst} m_t^*=\lceil \log (n\cdot 2^{-\varepsilon|t-t_1(n)|})
\rceil, \quad k_{t,m}=\lceil n\cdot 2^{-\varepsilon(|t-t_1(n)|+
|m-m_t^*|)}\rceil.
\end{align}
Then $k_{t,m}\in \N$ and $\sum \limits _{t=t_0}^{t_*(n)-1}\sum
\limits _{m=0}^\infty (k_{t,m}-1)
\underset{\mathfrak{Z}_0,\varepsilon}{\lesssim} n$.

{\bf Case $p>q$.} From H\"{o}lder's inequality it follows that for
$f\in B\hat X_p(\Omega)$, $m>m_t$
$$
\|P_{t,m+1}f-P_{t,m}f\|_{Y_q(G_{m,t})}
\stackrel{(\ref{fpttm2})}{\underset{\mathfrak{Z}_0}{\lesssim}}
$$
$$
\lesssim 2^{-\mu_*k_*t-\delta_*(m-m_t)}u_*(2^{k_*t})({\rm card}\,
\hat T_{t,m})^{\frac 1q-\frac 1p} \stackrel{(\ref{cttm}),
(\ref{card_et})}{\underset{\mathfrak{Z}_0}{\lesssim}}
2^{-\mu_*k_*t-\delta_*(m-m_t)}u_*(2^{k_*t})\cdot 2^{m\left(\frac
1q-\frac 1p\right)}
\stackrel{(\ref{mt_def})}{\underset{\mathfrak{Z}_0}{\lesssim}}
$$
$$
\lesssim 2^{-\mu_*k_*t} u_*(2^{k_*t}) \cdot
\overline{\nu}_t^{\delta_*} \cdot 2^{m\left(-\delta_*+\frac
1q-\frac 1p\right)}.
$$
Further, for $f\in B\hat X_p(\Omega)$ by H\"{o}lder's inequality
we het
$$
\|P_{t,m_t+1}f-Q_tf\|_{Y_q(G_t)} \le
\|f-P_{t,m_t+1}f\|_{Y_q(G_t)}+\|f-Q_tf\|_{Y_q(G_t)}
\stackrel{(\ref{cttm}),(\ref{fpttm2}),(\ref{fqtf})}{\underset{\mathfrak{Z}_0}{\lesssim}}
$$
$$
\lesssim 2^{-\mu_*k_*t}u_*(2^{k_*t})\cdot 2^{m_t\left(\frac
1q-\frac 1p\right)}+2^{-\mu_*k_*t} u_*(2^{k_*t})
\overline{\nu}_t^{\frac 1q-\frac 1p}
\stackrel{(\ref{mt_def})}{\underset{\mathfrak{Z}_0}{ \lesssim}}
2^{-\mu_*k_*t}u_*(2^{k_*t}) \overline{\nu}_t^{\frac 1q-\frac 1p}.
$$
Thus,
\begin{align}
\label{fff} \|P_{t,m+1}f-P_{t,m}f\|_{Y_q(G_{m,t})}
\underset{\mathfrak{Z}_0}{\lesssim} 2^{-\mu_*k_*t} u_*(2^{k_*t})
\cdot \overline{\nu}_t^{\delta_*} \cdot 2^{m\left(-\delta_*+\frac
1q-\frac 1p\right)},
\end{align}
\begin{align}
\label{ggg} \|P_{t,m_t+1}f-Q_tf\|_{Y_q(G_t)}
\underset{\mathfrak{Z}_0}{\lesssim} 2^{-\mu_*k_*t}u_*(2^{k_*t})
\overline{\nu}_t^{\frac 1q-\frac 1p}.
\end{align}

From (\ref{stm}) we get that $\overline{s}_{t,m_t+1}=\lceil
C_1\cdot 2^{m_t+1}\rceil \stackrel{(\ref{mt_def})}{\le} \lceil
4C_1\overline{\nu}_t\rceil$. Similarly as in Lemma
\ref{sum_qt_est} we prove that
\begin{align}
\label{nu_sum}\begin{array}{c} \sum \limits _{t=t_0}^{t_*(n)-1}
2^{-\mu_*k_*t}u_*(2^{k_*t}) \overline{\nu}_t^{\frac 1q-\frac 1p}
e_{k_t}(I_{\tilde{s}_{t,m_t+1}}: l_q^{\tilde{s}_{t,m_t+1}}
\rightarrow l_q^{\tilde{s}_{t,m_t+1}})
\stackrel{(\ref{tstm}),(\ref{stm})}{\le}
\\
\le \sum \limits _{t=t_0}^{t_*(n)-1} 2^{-\mu_*k_*t}u_*(2^{k_*t})
\overline{\nu}_t^{\frac 1q-\frac 1p}
e_{k_t}(I_{\overline{s}_{t,m_t+1}}: l_q^{\overline{s}_{t,m_t+1}}
\rightarrow l_q^{\overline{s}_{t,m_t+1}})
\underset{\mathfrak{Z}_0}{ \lesssim}\\ \lesssim
n^{-\mu_*\beta_*+\frac 1q-\frac 1p} \varphi_*^{-\mu_*}(n)
u_*(n^{\beta_*}\varphi_*(n)).
\end{array}
\end{align}

Let $\varepsilon>0$ be sufficiently small. From Theorem
\ref{shutt_trm} it follows that for $m\le m_t^*$
$$
e_{k_{t,m}}(I_{s_{t,m}}: l_q^{s_{t,m}} \rightarrow l_q^{s_{t,m}})
\stackrel{(\ref{tstm}),(\ref{stm})}{\le}
e_{k_{t,m}}(I_{\overline{s}_{t,m}}: l_q^{\overline{s}_{t,m}}
\rightarrow l_q^{\overline{s}_{t,m}})
\stackrel{(\ref{stm}),(\ref{mtst})}{\underset{p,q}{\lesssim}}
2^{-\frac{k_{t,m}} {\overline{s}_{t,m}}}.
$$
In addition, there exists $l_*=l_*(\mathfrak{Z}_0)$ such that for
any $t_0\le t<t_*(n)$, $0\le r<l_*$ the sequence
$\left\{\frac{k_{t,l_*m+r}}{\overline{s}_{t,l_*m+r}}\right\}_{l_*m+r\le
m_t^*}$ decreases not slower than some geometric progression. If
$m>m_t^*$, then
$$e_{k_{t,m}}(I_{s_{t,m}}: l_q^{s_{t,m}} \rightarrow
l_q^{s_{t,m}})\stackrel{(\ref{tstm}),(\ref{stm})}{\le}
e_{k_{t,m}}(I_{\overline{s}_{t,m}}: l_q^{\overline{s}_{t,m}}
\rightarrow l_q^{\overline{s}_{t,m}}) \underset{p,q}{\lesssim}1.$$
By Remark \ref{1511}, (\ref{nu_t_k}) holds. Hence,
$$
S_1 \stackrel{(\ref{slimtt0}),(\ref{fff}),
(\ref{ggg})}{\underset{\mathfrak{Z}_0}{\lesssim}} \sum \limits
_{t=t_0}^{t_*(n)-1} \sum \limits _{m=m_t+1}^\infty 2^{-\mu_*k_*t}
u_*(2^{k_*t}) \cdot \overline{\nu}_t^{\delta_*} \cdot
2^{m\left(-\delta_*+\frac 1q-\frac
1p\right)}e_{k_{t,m}}(I_{s_{t,m}}: l_q^{s_{t,m}} \rightarrow
l_q^{s_{t,m}}) +
$$
$$
+\sum \limits _{t=t_0}^{t_*(n)-1} 2^{-\mu_*k_*t}u_*(2^{k_*t})
\overline{\nu}_t^{\frac 1q-\frac 1p}
e_{k_{t}}(I_{\tilde{s}_{t,m_t+1}}:
l^{\tilde{s}_{t,m_t+1}}_{\sigma_{p,q}} \rightarrow
l_q^{\tilde{s}_{t,m_t+1}}) \stackrel{(\ref{delpq}),
(\ref{nu_sum})}{\underset{\mathfrak{Z}_0,\varepsilon}{\lesssim}}
$$
$$
\lesssim\sum \limits _{t=t_0}^{t_*(n)-1} 2^{-\mu_*k_*t}
u_*(2^{k_*t}) \cdot \overline{\nu}_t^{\delta_*} \cdot
2^{m_t^*\left(-\delta_*+\frac 1q-\frac
1p\right)}+n^{-\mu_*\beta_*+\frac 1q-\frac 1p}
\varphi_*^{-\mu_*}(n) u_*(n^{\beta_*}\varphi_*(n))
\stackrel{(\ref{nu_t_k}),(\ref{mtst})}{\underset{\mathfrak{Z}_0,\varepsilon}
{\lesssim}}
$$
$$
\lesssim \sum \limits _{t=t_0}^{t_*(n)-1} 2^{-\mu_*k_*t}
u_*(2^{k_*t}) \cdot 2^{\delta_*\gamma_*k_*t}
\psi_*^{\delta_*}(2^{k_*t}) \cdot \left(n\cdot
2^{-\varepsilon|t-t_1(n)|}\right)^{-\delta_*+\frac 1q-\frac 1p}+
$$
$$
+n^{-\mu_*\beta_*+\frac 1q-\frac 1p} \varphi_*^{-\mu_*}(n)
u_*(n^{\beta_*}\varphi_*(n))=:S_1'.
$$
Recall that $\beta_*=\frac{1}{\gamma_*}$. If
$\delta_*<\mu_*\beta_*$, then for sufficiently small
$\varepsilon>0$ we have
$S_1'\stackrel{(\ref{t11})}{\underset{\mathfrak{Z}_0} {\lesssim}}
n^{-\delta_*+\frac 1q-\frac 1p}$. If $\delta_*>\mu_*\beta_*$, then
for sufficiently small $\varepsilon>0$ we get
$$
S_1'\stackrel{(\ref{t11})}{\underset{\mathfrak{Z}_0} {\lesssim}}
2^{-\mu_*k_*t_*(n)} u_*(2^{k_*t_*(n)})
\overline{\nu}_{t_*(n)}^{\delta_*}\cdot n^{-\delta_*+\frac
1q-\frac 1p} +n^{-\mu_*\beta_*+\frac 1q-\frac 1p}
\varphi_*^{-\mu_*}(n) u_*(n^{\beta_*}\varphi_*(n))
\stackrel{(\ref{nu_t_est1})}{\underset{\mathfrak{Z}_0} {\lesssim}}
$$
$$
\lesssim n^{-\mu_*\beta_*+\frac 1q-\frac 1p}\sigma_*(n).
$$

This completes the proof for $p>q$.

{\bf Case $p\le q$.} Let $f\in B\hat X_p(\Omega)$. Then for $m\le
m_t$
\begin{align}
\label{ptm1fplq} \|P_{t,m+1}f-P_{t,m}f\|_{p,q,\tilde T_{t,m}}
\stackrel{(\ref{fpttm1})}{\underset{\mathfrak{Z}_0}{\lesssim}}
2^{-\lambda_*k_*t} u_*(2^{k_*t})\le 2^{-\lambda_*k_*t}
u_*(2^{k_*t})\cdot 2^{-\delta_*(m-m_t)},
\end{align}
and for $m>m_t$
\begin{align}
\label{ptmtil} \|P_{t,m+1}f-P_{t,m}f\|_{p,q,\tilde T_{t,m}}
\stackrel{(\ref{mu_ge_lambda}),(\ref{fpttm2})}{\underset{\mathfrak{Z}_0}{\lesssim}}
2^{-\lambda_*k_*t} u_*(2^{k_*t})\cdot 2^{-\delta_*(m-m_t)}.
\end{align}
By Theorem \ref{shutt_trm}, we have for $m\le m^*_t$
$$
e_{k_{t,m}}(I_{s_{t,m}}:l_p^{s_{t,m}} \rightarrow l_q^{s_{t,m}})
\stackrel{(\ref{tstm}),(\ref{stm})}{\le}
e_{k_{t,m}}(I_{\overline{s}_{t,m}}: l_p^{\overline{s}_{t,m}}
\rightarrow l_q^{\overline{s}_{t,m}})
\stackrel{(\ref{stm}),(\ref{mtst})}{ \underset{p,q}{\lesssim}}
2^{-\frac{k_{t,m}}{\overline{s}_{t,m}}} \overline{s}_{t,m}^{\frac
1q-\frac 1p}
$$
(moreover, there exists $l_*=l_*(\mathfrak{Z}_0)\in \N$ such that
for any $0\le r< l_*$, $t_0\le t<t_*(n)$ the sequence
$\left\{\frac{k_{t,l_*m+r}}{\overline{s}_{t,l_*m+r}}\right\}_{0\le
l_*m+r\le m_t^*}$ decreases not slower than some geometric
progression), and for $m> m_t^*$ we have
$$
e_{k_{t,m}}(I_{s_{t,m}}:l_p^{s_{t,m}} \rightarrow l_q^{s_{t,m}})
\underset{p,q}{\lesssim} \min \left\{\frac{\log
\left(1+\frac{\overline{s}_{t,m}}{k_{t,m}}\right)}{k_{t,m}}, \,
1\right\}^{\frac 1p-\frac 1q}.
$$
This implies that
$$
S_2 \stackrel{(\ref{slimtt01}),(\ref{ptm1fplq}),
(\ref{ptmtil})}{\underset {\mathfrak{Z}_0} {\lesssim}} \sum
\limits _{t=t_0}^{t_*(n)-1} \sum \limits_{m=0}^{m^*_t}
2^{-\lambda_*k_*t} u_*(2^{k_*t})\cdot 2^{-\delta_*(m-m_t)}
2^{-\frac{k_{t,m}}{\overline{s}_{t,m}}} \overline{s}_{t,m}^{\frac
1q-\frac 1p}+
$$
$$
+\sum \limits _{t=t_0}^{t_*(n)-1} \sum \limits_{m=m^*_t+1}^\infty
2^{-\lambda_*k_*t} u_*(2^{k_*t})\cdot 2^{-\delta_*(m-m_t)} \min
\left\{\frac{\log
\left(1+\frac{\overline{s}_{t,m}}{k_{t,m}}\right)}{k_{t,m}}, \,
1\right\}^{\frac 1p-\frac 1q}
\stackrel{(\ref{mt_def}),(\ref{stm}),
(\ref{mtst})}{\underset{\mathfrak{Z}_0}{\lesssim}}
$$
$$
\lesssim\sum \limits _{t=t_0}^{t_*(n)-1} 2^{-\lambda_*k_*t}
u_*(2^{k_*t}) \cdot 2^{-\delta_*m_t^*}
\overline{\nu}_t^{\delta_*}k_{t,m_t^*}^{\frac 1q-\frac 1p}
\stackrel{(\ref{mtst})}{\underset{\mathfrak{Z}_0}{\lesssim}}
$$
$$
\lesssim \sum \limits _{t=t_0}^{t_*(n)-1} 2^{-\lambda_*k_*t}
u_*(2^{k_*t})(n\cdot 2^{-\varepsilon|t-t_1(n)|})^{-\delta_*+\frac
1q-\frac 1p}\overline{\nu}_t^{\delta_*}=:S'_2.
$$
If (\ref{nu_t_k}) holds, then $\overline{\nu}_t=
2^{\gamma_*k_*t}\psi_*(2^{k_*t})$. Hence, for
$\delta_*<\lambda_*\beta_*$ we have $S_2'
\stackrel{(\ref{t12})}{\underset{\mathfrak{Z}_0}{\lesssim}}
n^{-\delta_*+\frac 1q-\frac 1p}$, and for
$\delta_*>\lambda_*\beta_*$ we get
$$
S_2'\stackrel{(\ref{t12})}{\underset{\mathfrak{Z}_0}{\lesssim}}
2^{-\lambda_*k_*t_*(n)} u_*(2^{k_*t_*(n)}) n^{-\delta_* +\frac
1q-\frac 1p} \overline{\nu}_{t_*(n)}^{\delta_*}
\stackrel{(\ref{nu_t_est1})}{\underset{\mathfrak{Z}_0}{\lesssim}}
n^{-\lambda_*\beta_*+\frac 1q-\frac 1p} \sigma_*(n).
$$
If (\ref{nu_t_k1}) holds, then $\overline{\nu}_t=2^{\gamma_*2^t}
\psi_*(2^{2^t})$, $k_*=1$ and
$$
S_2'\stackrel{(\ref{t_st_n}),(\ref{nu_t_est2}),(\ref{t13})}{\underset{\mathfrak{Z}_0}{\lesssim}}
n^{\frac 1q-\frac 1p} (\log n)^{-\lambda_*} u_*(\log n)\le
n^{\frac 1q-\frac 1p} (\log n)^{-\lambda_*+\frac 1p-\frac 1q}
u_*(\log n).
$$
This completes the proof.
\end{proof}
Theorems \ref{trm1} and \ref{trm2} for $p\ge q$ follow from
(\ref{eklspt}), (\ref{t_st_st_n}), (\ref{f_q_t}), (\ref{f_p_t_m}),
(\ref{f_p_t_m1}) and Lemmas \ref{qt_st_n}, \ref{sum_qt_est},
\ref{ptm_sum}.

Consider the case $p<q$.

We set
\begin{align}
\label{t2n} t_2(n)=\left\{ \begin{array}{l} t_*(n), \quad \text{if
(\ref{nu_t_k}) holds, or if (\ref{nu_t_k1}) holds and
}\lambda_*>\frac 1p-\frac 1q,
\\ t_{**}(n), \quad \text{if (\ref{nu_t_k1}) holds and }
\lambda_*<\frac 1p-\frac 1q.\end{array} \right.
\end{align}

\begin{Lem}
\label{t_g_tstn} Let $p<q$. We set
\begin{align}
\label{ktm_def} k_{t,m} =\lceil n\cdot 2^{-\varepsilon|t-t_2(n)|
-\varepsilon(m-m_t)}\rceil, \quad t_*(n)\le t<t_{**}(n), \quad
m\ge m_t.
\end{align}
Then
\begin{align}
\label{sum_ktm} \sum \limits _{t_*(n)\le t<t_{**}(n),\, m\ge m_t}
(k_{t,m}-1) \underset{\mathfrak{Z}_0,\varepsilon} {\lesssim} n
\end{align}
and for sufficiently small $\varepsilon>0$ the following
assertions hold.
\begin{enumerate}
\item Let (\ref{nu_t_k}) hold. Then
$$
\sum \limits _{t=t_*(n)}^{t_{**}(n)-1} \sum \limits
_{m=m_t}^\infty e_{k_{t,m}}(P_{t,m+1}-P_{t,m}:\hat X_p(\Omega)
\rightarrow Y_q(G_t)) \underset{\mathfrak{Z}_0}{\lesssim}
n^{-\lambda_*\beta_*+\frac 1q-\frac 1p}
\varphi_*^{-\lambda_*}(n)u_*(n^{\beta_*} \varphi_*(n)).
$$
\item Let (\ref{nu_t_k1}) hold. Then
$$
\sum \limits _{t=t_*(n)}^{t_{**}(n)-1} \sum \limits
_{m=m_t}^\infty e_{k_{t,m}}(P_{t,m+1}-P_{t,m}:\hat X_p(\Omega)
\rightarrow Y_q(G_t)) \underset{\mathfrak{Z}_0}{\lesssim} n^{\frac
1q-\frac 1p}(\log n)^{-\lambda_*+\frac 1p-\frac 1q} u_*(\log n)
$$
for $\lambda_*>\frac 1p-\frac 1q$,
$$
\sum \limits _{t=t_*(n)}^{t_{**}(n)-1} \sum \limits
_{m=m_t}^\infty e_{k_{t,m}}(P_{t,m+1}-P_{t,m}:\hat X_p(\Omega)
\rightarrow Y_q(G_t)) \underset{\mathfrak{Z}_0}{\lesssim}
n^{-\lambda_*} u_*(n)
$$
for $\lambda_*<\frac 1p-\frac 1q$.
\end{enumerate}
\end{Lem}
\begin{proof}
The relation (\ref{sum_ktm}) follows from (\ref{ktm_def}).

Let $\hat T_{t,m}$, $s_{t,m}$ be defined by (\ref{hattm}),
(\ref{stmdef}). From (\ref{tstm}) it follows that $s_{t,m}\le
C_1\cdot 2^m$. We set $\overline{s}_{t,m}=\lceil C_1\cdot 2^m
\rceil$. Then
\begin{align}
\label{stmt} \overline{s}_{t,m_t} \stackrel{(\ref{mt_def})
}{\underset{\mathfrak{Z}_0}{\asymp}} \overline{\nu}_t.
\end{align}

By (\ref{mult_n}) and (\ref{atm_m1}), it suffices to estimate
$$
\sum \limits _{t=t_*(n)}^{t_{**}(n)-1} \sum \limits
_{m=m_t}^\infty \|P_{t,m+1} -P_{t,m}\| _{\hat X_p(\Omega)
\rightarrow Y_{p,q,\hat
T_{t,m}}(G_t)}e_{k_{t,m}}(I_{\overline{s}_{t,m}}:l_p^{\overline{s}_{t,m}}\rightarrow
l_q^{\overline{s}_{t,m}})=:S.
$$
From (\ref{mu_ge_lambda}), (\ref{fpttm1}), (\ref{fpttm2}) it
follows that
$$
\|P_{t,m+1} -P_{t,m}\| _{\hat X_p(\Omega) \rightarrow Y_{p,q,\hat
T_{t,m}}(G_t)} \underset{\mathfrak{Z}_0}{\lesssim}
2^{-\lambda_*k_*t} u_*(2^{k_*t}) \cdot 2^{-\delta_*(m-m_t)}.
$$
Since $\overline{s}_{t,m}\stackrel{(\ref{stm})}{\ge} C_1\cdot
2^{m_t}\stackrel{(\ref{mt_def})}{\underset{\mathfrak{Z}_0}{\gtrsim}}
\overline{\nu}_t \stackrel{(\ref{nu_t_k}),
(\ref{nu_t_k1})}{\underset{\mathfrak{Z}_0}{\gtrsim}}
\overline{\nu}_{t_*(n)} \stackrel{(\ref{t_st_n})}{\ge} n$ and
$k_{t,m} \stackrel{(\ref{ktm_def})}{\le} n$, we get by Theorem
\ref{shutt_trm} that
$$
e_{k_{t,m}}(I_{\overline{s}_{t,m}}:l_p^{\overline{s}_{t,m}}\rightarrow
l_q^{\overline{s}_{t,m}})\underset{\mathfrak{Z}_0}{\lesssim} \min
\left\{\frac{\log
\left(1+\frac{\overline{s}_{t,m}}{k_{t,m}}\right)}{k_{t,m}}, \,
1\right\}^{\frac 1p-\frac 1q}.
$$
Hence, for sufficiently small $\varepsilon>0$ we have
$$
S\underset{\mathfrak{Z}_0}{\lesssim} \sum \limits
_{t=t_*(n)}^{t_{**}(n)-1} 2^{-\lambda_*k_*t} u_*(2^{k_*t})
\left\{\frac{\log
\left(1+\frac{\overline{s}_{t,m_t}}{k_{t,m_t}}\right)}{k_{t,m_t}}\right\}^{\frac
1p-\frac
1q}\stackrel{(\ref{ktm_def}),(\ref{stmt})}{\underset{\mathfrak{Z}_0}{\lesssim}}
$$
$$
\lesssim \sum \limits _{t=t_*(n)}^{t_{**}(n)-1} 2^{-\lambda_*k_*t}
u_*(2^{k_*t}) \cdot n^{\frac 1q-\frac 1p} \cdot
2^{\varepsilon|t-t_2(n)|\left(\frac 1p-\frac 1q\right)} \left(\log
\left( 1+\frac{\overline{\nu}_t}{\lceil n\cdot
2^{-\varepsilon|t-t_2(n)|}\rceil}\right)\right)^{\frac 1p-\frac
1q}.
$$
This together with (\ref{t_st_n}), (\ref{nu_t_est1}),
(\ref{nu_t_est2}), (\ref{t2n}) yields the desired estimates for
sufficiently small $\varepsilon>0$.
\end{proof}

\begin{Lem}
\label{entr_pmtqt} Let $p<q$. Then there exists a sequence
$\{\tilde k_t\}_{t_*(n)\le t<t_{**}(n)}\subset \N$ such that
\begin{align}
\label{stkt1} \sum \limits _{t=t_*(n)}^{t_{**}(n)-1} (\tilde
k_t-1) \underset{\mathfrak{Z}_0}{\lesssim} n
\end{align}
and the following assertions hold.
\begin{enumerate}
\item Let (\ref{nu_t_k}) hold. Then
$$
\sum \limits _{t=t_*(n)}^{t_{**}(n)-1} e_{\tilde
k_t}(P_{t,m_t}-Q_t:\hat X_p(\Omega) \rightarrow Y_q(G_t))
\underset {\mathfrak{Z}_0}{\lesssim} n^{-\lambda_*\beta_*+\frac
1q-\frac 1p} u_*(n^{\beta_*} \varphi_*(n))
\varphi_*^{-\lambda_*}(n).
$$
\item Let (\ref{nu_t_k1}) hold. Then for $\lambda_* >\frac 1p-\frac
1q$ we have
$$
\sum \limits _{t=t_*(n)}^{t_{**}(n)-1} e_{\tilde
k_t}(P_{t,m_t}-Q_t:\hat X_p(\Omega) \rightarrow Y_q(G_t))
\underset {\mathfrak{Z}_0}{\lesssim} n^{\frac 1q-\frac 1p} (\log
n)^{-\lambda_*+\frac 1p-\frac 1q} u_*(\log n),
$$
and for $\lambda_*<\frac 1p-\frac 1q$ we have
$$
\sum \limits _{t=t_*(n)}^{t_{**}(n)-1} e_{\tilde
k_t}(P_{t,m_t}-Q_t:\hat X_p(\Omega) \rightarrow Y_q(G_t))
\underset {\mathfrak{Z}_0}{\lesssim} n^{-\lambda_*} u_*(n).
$$
\end{enumerate}
\end{Lem}
We shall use Theorem \ref{lifs_sta} and the following assertion.

\begin{Lem}
\label{lemma_o_razb_dereva1} {\rm (see  \cite{vas_john}).} Let
$({\cal T}, \, \xi_*)$ be a tree with finite vertex set, let
\begin{align}
\label{cardvvvk} {\rm card}\, {\bf V}_1(\xi)\le k, \;\; \xi\in
{\bf V}({\cal T}),
\end{align}
and let the mapping $\Phi :2^{{\bf V}(\cal T)}\rightarrow \R_+$
satisfy the following condition:
\begin{align}
\label{prop_psi} \Phi(V_1\cup V_2)\ge \Phi(V_1)+\Phi(V_2), \; V_1,
\, V_2\subset {\bf V}({\cal T}), \;\; V_1\cap V_2=\varnothing,
\end{align}
$\Phi({\bf V}({\cal T}))>0$. Then there is a number $C(k)>0$ such
that for any $n\in \N$ there exists a partition $\mathfrak{S}_n$
of the tree ${\cal T}$ into at most $C(k)n$ subtrees ${\cal T}_j$,
which satisfies the following conditions:
\begin{enumerate}
\item $\Phi({\bf V}({\cal T}_j))\le \frac{(k+2)\Phi({\bf V}({\cal T}))}{n}$
for any $j$ such that ${\rm card}\, {\bf V}({\cal T}_j)\ge 2$;
\item if $m\le 2n$, then each element of
$\mathfrak{S}_n$ intersects with at most $C(k)$ elements of
$\mathfrak{S}_m$.
\end{enumerate}

\end{Lem}

\renewcommand{\proofname}{\bf Proof of Lemma \ref{entr_pmtqt}}
\begin{proof}
{\bf Step 1.} We define the numbers $t_2(n)$ by (\ref{t2n});
$\varepsilon=\varepsilon(\mathfrak{Z}_0)$ will be chosen later.
The sequence $\tilde k_t(n)$ is defined so that
\begin{align}
\label{til_kt_est} \tilde k_t(n)-1
\underset{\mathfrak{Z}_0}{\lesssim}
 n\cdot 2^{-\varepsilon|t-t_2(n)|} \quad \text{if
 }(\ref{nu_t_k})\text{ holds},
\end{align}
\begin{align}
\label{til_kt_est1} \tilde k_t(n)-1
\underset{\mathfrak{Z}_0}{\lesssim} \max \left\{n\cdot
2^{-\varepsilon|t-t_2(n)|}, \, 2^t\right\} \quad \text{if
 }(\ref{nu_t_k1})\text{ holds}.
\end{align}
This together with (\ref{nu_t_est2}) implies (\ref{stkt1}).

Further we consider $t_*(n)\le t<t_{**}(n)$.

{\bf Step 2.} Let ${\bf T}$ be a partition of $\Gamma_t$ into
subtrees. Then ${\bf T}= \{{\cal A}_{t,i,s}\}_{i\in \hat J_t, \,
s\in I_{t,i}}$, where $\{{\cal A}_{t,i,s}\}_{s\in I_{t,i}}$ --- is
a partition of ${\cal A}_{t,i}$. Denote by $\hat \xi_{t,i,s}$ the
root of ${\cal A}_{t,i,s}$.

Define the operator $\hat P_{\mathbf{T}}:\hat X_p(\Omega)
\rightarrow Y_q(\Omega)$ as follows. We set $\hat
P_{\mathbf{T}}f|_{\Omega \backslash G_t}=0$,
\begin{align}
\label{pt1} \hat P_{\mathbf{T}}f|_{\Omega _{{\cal A}_{t,i,s}}}=0,
\quad \text{if} \quad \hat \xi_{t,i,s}=\hat \xi_{t,i}
\end{align}
\begin{align}
\label{pt2} \hat P_{\mathbf{T}}f|_{\Omega _{{\cal A}_{t,i,s}}}=
(P_{t,m_t}f-Q_tf)|_{\Omega _{{\cal A}_{t,i,s}}} \quad \text{if}
\quad {\bf V}({\cal A}_{t,i,s})=\{\hat \xi_{t,i,s}\};
\end{align}
in other cases we set
\begin{align}
\label{pt3} \hat P_{\mathbf{T}}f|_{\Omega _{{\cal A}_{t,i,s}}}=
(P_{\Omega _{{\cal A}_{\hat \xi_{t,i,s}}}}f-Q_tf)|_{\Omega _{{\cal
A}_{t,i,s}}}
\end{align}
(see Assumption \ref{sup1}). Let $T=\{\Omega_{{\cal A}'}\}_{{\cal
A}'\in \mathbf{T}}$. Notice that
\begin{align}
\label{hptfsto} \hat P_{\mathbf{T}}f \in {\cal S}_{{T}}(\Omega).
\end{align}

Let $f\in B\hat X_p(\Omega)$. If $\hat \xi_{t,i,s}=\hat
\xi_{t,i}$, then
\begin{align}
\label{ptmt1} \|P_{t,m_t}f-Q_tf-\hat P_{\mathbf{T}}f\|_{Y_q(\Omega
_{{\cal A}_{t,i,s}})}
\stackrel{(\ref{ftp_ti_mod}),(\ref{qtf_x}),(\ref{fpttm1}),
(\ref{pt1})}{\underset{\mathfrak{Z}_0}{\lesssim}}
2^{-\lambda_*k_*t} u_*(2^{k_*t}) \|f\|_{X_p(\Omega _{{\cal
A}_{t,i,s}})}.
\end{align}
If ${\bf V}({\cal A}_{t,i,s})=\{\hat \xi_{t,i,s}\}$, then
\begin{align}
\label{ptmt2} \|P_{t,m_t}f-Q_tf-\hat
P_{\mathbf{T}}f\|_{Y_q(\Omega_{{\cal
A}_{t,i,s}})}\stackrel{(\ref{pt2})}{=}0.
\end{align}
In other cases
\begin{align}
\label{ptmt3} \|P_{t,m_t}f-Q_tf-\hat
P_{\mathbf{T}}f\|_{Y_q(\Omega)}
\stackrel{(\ref{f_pom_f}),(\ref{w_s_2}),(\ref{fpttm1}),(\ref{pt3})}
{\underset{\mathfrak{Z}_0}{\lesssim}} 2^{-\lambda_*k_*t}
u_*(2^{k_*t}) \|f\|_{X_p(\Omega _{{\cal A}_{t,i,s}})}.
\end{align}
We set $\mathbf{T}'=\{{\cal A}'\in \mathbf{T}:\; {\rm card}\, {\bf
V}({\cal A}')\ge 2\}$. Then for any $f \in B\hat X_p(\Omega)$
\begin{align}
\label{n_ptmtqt} \|P_{t,m_t}f-Q_tf-\hat P_{\mathbf{T}}f\|
_{Y_q(G_t)} \stackrel{(\ref{ptmt1}),
(\ref{ptmt2}),(\ref{ptmt3})}{\underset{\mathfrak{Z}_0}{\lesssim}}
2^{-\lambda_*k_*t} u_*(2^{k_*t}) \left(\sum \limits _{{\cal A}'\in
\mathbf{T}'} \|f\|_{X_p(\Omega_{{\cal A}'})}^q\right)^{1/q}.
\end{align}

{\bf Step 3.} Let $r\in \N$,
\begin{align}
\label{r_le_nu4} r\le \frac{\overline{\nu}_t}{4}.
\end{align}
Denote by ${\cal N}_r$ the family of partitions $\mathbf{T}$ of
the graph $\Gamma_t$ into subtrees such that ${\rm card}\, \{{\cal
A}_{t,i,s}\in \mathbf{T}:\; \hat\xi_{t,i,s}\ne \hat \xi_{t,i}\}\le
r$. The number $|{\cal N}_r|$ can be estimated from above by the
number of choices of sets of vertices $\hat \xi_{t,i,s}\ne \hat
\xi_{t,i}$ in ${\bf V}(\Gamma_t)$. Therefore,
\begin{align}
\label{cnr} |{\cal N}_r|\le \sum \limits _{m=0}^r C_{|{\bf
V}(\Gamma_t)|}^m\stackrel{(\ref{nu_t_k}),
(\ref{nu_t_k1})}{\le}\sum \limits _{m=0}^r C_{\lceil
c_3\overline{\nu}_t\rceil}^m \stackrel{(\ref{r_le_nu4})}{\lesssim}
C^r_{\lceil c_3\overline{\nu}_t\rceil} \le \left(\frac{e\lceil
c_3\overline{\nu}_t\rceil}{r}\right)^r.
\end{align}

If (\ref{nu_t_k}) holds, then we set
\begin{align}
\label{rt_n} r_t=\lceil n\cdot
2^{-2\varepsilon|t-t_2(n)|-c}\rceil, \quad n\ge N(\mathfrak{Z}_0),
\end{align}
where $c=c(\mathfrak{Z}_0)$, $N(\mathfrak{Z}_0)$ are such that
$r_t\le \frac{\overline{\nu}_t}{4}$ for $n\ge N(\mathfrak{Z}_0)$.
Then in the case $n\cdot 2^{-2\varepsilon|t-t_2(n)|-c}\ge 1$ we
have
$$
\log|{\cal N}_{r_t-1}|\le \log|{\cal N}_{r_t}|
\stackrel{(\ref{cnr})} {\underset {\mathfrak{Z}_0}{\lesssim}} r_t
\log \frac {c_3e \overline{\nu}_t}{r_t}
\stackrel{(\ref{t2n}),(\ref{rt_n})}{\underset
{\mathfrak{Z}_0}{\lesssim}}
$$
$$
\lesssim n\cdot 2^{-2\varepsilon(t-t_*(n))} \log
\left(\frac{c_3e\overline{\nu}_{t_*(n)}}{n\cdot 2^
{-2\varepsilon(t-t_*(n))-c}}\cdot \frac{\overline {\nu}_t}
{\overline{\nu}_{t_*(n)}}\right) \stackrel{(\ref{nu_t_k}),
(\ref{sum_l_est}),(\ref{nu_t_est1})}{\underset{\mathfrak{Z}_0}{\lesssim}}
$$
$$
\lesssim n\cdot 2^{-2\varepsilon(t-t_*(n))} \log (2^
{2\varepsilon(t-t_*(n))} \cdot 2^{2\gamma_*k_*(t-t_*(n))}\cdot e)
\underset{\mathfrak{Z}_0}{\lesssim}
$$
$$
\lesssim n\cdot 2^{-2\varepsilon(t-t_*(n))} \cdot
(2\gamma_*k_*+2\varepsilon)(t-t_*(n)+1)\underset{\mathfrak{Z}_0,
\varepsilon}{\lesssim} n\cdot2^{-\varepsilon(t-t_*(n))};
$$
if $n\cdot 2^{-2\varepsilon|t-t_2(n)|-c}< 1$, then $|{\cal
N}_{r_t-1}| \stackrel{(\ref{cnr})}{=} C^0_{|{\bf V}(\Gamma_t)|}=1$
and $\log |{\cal N}_{r_t-1}|=0$. Therefore,
\begin{align}
\label{log_nt} \log|{\cal N}_{r_t-1}|
\underset{\mathfrak{Z}_0,\varepsilon}{\lesssim}
n\cdot2^{-\varepsilon(t-t_*(n))}.
\end{align}

If (\ref{nu_t_k1}) holds, then we set
\begin{align}
\label{rt_n1} r_t=\left\lceil \frac{n\cdot
2^{-\varepsilon|t-t_2(n)|}}{2^{t+c}}\right\rceil,
\end{align}
where $c=c(\mathfrak{Z}_0)$ is such that $r_t\le
\frac{\overline{\nu}_t}{4}$. Then
$$
\log|{\cal N}_{r_t-1}|\le \log|{\cal N}_{r_t}|
\stackrel{(\ref{cnr})}{\underset{\mathfrak{Z}_0}{\lesssim}} r_t
\log \frac{ec_3\overline{\nu}_t}{r_t}\le r_t \log
(c_3e\overline{\nu}_t)
\stackrel{(\ref{nu_t_k1}),(\ref{sum_l_est})}{\underset
{\mathfrak{Z}_0}{\lesssim}}
$$
$$
\lesssim r_t \cdot 2^t\underset {\mathfrak{Z}_0}{\lesssim} \max
\left\{n\cdot2^{-\varepsilon|t-t_2(n)|}, \, 2^t\right\};
$$
i.e.,
\begin{align}
\label{log_nt1} \log|{\cal N}_{r_t-1}|
\underset{\mathfrak{Z}_0,\varepsilon}{\lesssim} \max
\left\{n\cdot2^{-\varepsilon|t-t_2(n)|}, \, 2^t\right\}.
\end{align}

{\bf Step 4.} Consider the tree $\tilde{\cal A}_t$ with vertex set
${\bf V}({\cal A})\backslash {\bf V}(\tilde {\Gamma}_{t+1})$. For
$f\in B\hat X_p(\Omega)$ we define the function $\Phi_f:2^{{\bf
V}(\tilde{\cal A}_t)} \rightarrow \R_+$ by
\begin{align}
\label{phif} \Phi_f({\bf W}) =\sum \limits _{\xi\in {\bf W}\cap
{\bf V}(\Gamma_t)} \|f\|^p_{X_p(\hat F(\xi))}.
\end{align}
Then $\Phi_f({\bf W}_1\sqcup {\bf W}_2)=\Phi_f({\bf
W}_1)+\Phi_f({\bf W}_2)$. From (\ref{c_v1_a}) and Lemma
\ref{lemma_o_razb_dereva1} it follows that there exists a sequence
of partitions $\{\mathbf{T}_{f,t,l}\}_{0\le l\le \log r_t}$ of the
tree $\tilde{\cal A}_t$ such that
\begin{align}
\label{ctftl} {\rm card}\, \mathbf{T}_{f,t,l}\le r_t\cdot 2^{-l},
\end{align}
\begin{align}
\label{ctinter} {\rm card}\, \{{\cal A}''\in \mathbf{T}_{f,t,l\pm
1}:\; {\bf V}({\cal A}'')\cap {\bf V}({\cal A}')\ne \varnothing\}
\underset{\mathfrak{Z}_0}{\lesssim} 1, \quad {\cal A}'\in
\mathbf{T}_{f,t,l},
\end{align}
and for any subtree ${\cal A}'\in \mathbf{T}_{f,t,l}$ such that
${\rm card}\, {\bf V}({\cal A}')\ge 2$ the following estimate
holds:
\begin{align}
\label{phi_f_va} \Phi_f({\bf V}({\cal A}'))
\underset{\mathfrak{Z}_0}{\lesssim} \frac{2^l}{r_t}.
\end{align}
Here we may assume that
\begin{align}
\label{ctftlrt1} {\rm card}\, \mathbf{T}_{f,t,\lfloor\log
r_t\rfloor}=1.
\end{align}
Denote $\mathbf{T}_{f,t,l}^*:=\mathbf{T}_{f,t,l}|_{\Gamma_t}$.

We have
\begin{align}
\label{tft0nrt}
\mathbf{T}_{f,t,0}^*=\mathbf{T}_{f,t,0}|_{\Gamma_t}\in {\cal
N}_{r_t-1},
\end{align}
\begin{align}
\label{at} \sup _{f\in B\hat X_p(\Omega)}  \|P_{t,m_t}f-Q_tf-\hat
P_{\mathbf{T}^*_{f,t,0}}f\| _{Y_q(G_t)}
\stackrel{(\ref{n_ptmtqt}), (\ref{phif}),
(\ref{ctftl}),(\ref{phi_f_va})}{\underset{\mathfrak{Z}_0}{\lesssim}}
2^{-\lambda_*k_*t} u_*(2^{k_*t}) r_t^{\frac 1q-\frac 1p}=:A_t.
\end{align}

{\bf Step 5.} Let (\ref{nu_t_k}) hold. Then for sufficiently small
$\varepsilon>0$ we get
$$
\sum \limits _{t=t_*(n)}^{t_{**}(n)-1} A_t\stackrel{(\ref{t2n}),
(\ref{rt_n}),(\ref{at})}{\underset{\mathfrak{Z}_0}{\lesssim}} \sum
\limits _{t=t_*(n)}^{t_{**}(n)-1}2^{-\lambda_*k_*t} u_*(2^{k_*t})
(n\cdot 2^{-2\varepsilon(t-t_*(n))})^{\frac 1q-\frac
1p}\stackrel{(\ref{nu_t_est1})}{\underset{\mathfrak{Z}_0}{\lesssim}}
$$
$$
\lesssim n^{-\lambda_*\beta_*+\frac 1q-\frac 1p}u_*(n^{\beta_*}
\varphi_*(n)) \varphi_*^{-\lambda_*}(n);
$$
i.e.,
\begin{align}
\label{at_sum} \sum \limits _{t=t_*(n)}^{t_{**}(n)-1} A_t
\underset{\mathfrak{Z}_0}{\lesssim} n^{-\lambda_*\beta_*+\frac
1q-\frac 1p}u_*(n^{\beta_*} \varphi_*(n))
\varphi_*^{-\lambda_*}(n).
\end{align}
Let (\ref{nu_t_k1}) hold. Then
$$
\sum \limits _{t=t_*(n)}^{t_{**}(n)-1}
A_t\stackrel{(\ref{rt_n1}),(\ref{at})}{\underset{\mathfrak{Z}_0}{\lesssim}}
\sum \limits _{t=t_*(n)}^{t_{**}(n)-1} 2^{-\lambda_*t} u_*(2^{t})
(n\cdot 2^{-\varepsilon|t-t_2(n)|-t})^{\frac 1q-\frac 1p}=:A.
$$
In $\lambda_*>\frac 1p-\frac 1q$, then for sufficiently small
$\varepsilon>0$ we have
$$
A \stackrel{(\ref{t2n})}{\underset{\mathfrak{Z}_0}{\lesssim}}
2^{-\lambda_*t_*(n)}u_*(2^{t_*(n)})\cdot n^{\frac 1q-\frac
1p}\cdot 2^{\left(\frac 1p-\frac 1q\right)t_*(n)}
\stackrel{(\ref{nu_t_est2})}{\underset{\mathfrak{Z}_0}{\lesssim}}
n^{\frac 1q-\frac 1p} (\log n)^{-\lambda_*+\frac 1p-\frac 1q}
u_*(\log n);
$$
i.e.,
\begin{align}
\label{at_sum1} \sum \limits _{t=t_*(n)}^{t_{**}(n)-1} A_t
\underset{\mathfrak{Z}_0}{\lesssim} n^{\frac 1q-\frac 1p} (\log
n)^{-\lambda_*+\frac 1p-\frac 1q} u_*(\log n).
\end{align}
If $\lambda_*<\frac 1p-\frac 1q$, then for sufficiently small
$\varepsilon>0$
$$
A \stackrel{(\ref{t2n})}{\underset{\mathfrak{Z}_0}{\lesssim}}
2^{-\lambda_*t_{**}(n)}u_*(2^{t_{**}(n)})\cdot n^{\frac 1q-\frac
1p}\cdot 2^{\left(\frac 1p-\frac 1q\right)t_{**}(n)}
\stackrel{(\ref{nu_t_est2})}{\underset{\mathfrak{Z}_0}{\lesssim}}
n^{-\lambda_*} u_*(n);
$$
i.e.,
\begin{align}
\label{at_sum2} \sum \limits _{t=t_*(n)}^{t_{**}(n)-1} A_t
\underset{\mathfrak{Z}_0}{\lesssim} n^{-\lambda_*} u_*(n).
\end{align}

{\bf Step 6.} Let $0\le l\le \log r_t$, $\mathbf{T}_{t,l}
=\mathbf{T}_{\hat f,t,l}^*$ for some function $\hat f\in B\hat
X_p(\Omega)$. We set
\begin{align}
\label{ktl} \hat k_{t,l}=\lceil n\cdot
2^{-\varepsilon(|t-t_2(n)|+l)}\rceil.
\end{align}
Let us estimate the sum
$$
\sum \limits _{t=t_*(n)}^{t_{**}(n)-1} \sum \limits _{0\le l\le
\log r_t-1} e_{\hat k_{t,l}} (\hat P_{\mathbf{T}_{t,l}}-\hat
P_{\mathbf{T}_{t,l+1}}:\hat X_p(\Omega) \rightarrow Y_q(G_t))
$$
(notice that $\hat P_{\mathbf{T}_{t,\lfloor \log r_t\rfloor}}
\stackrel{(\ref{pt1})}{=} 0$ since ${\rm card}\, \mathbf{T}_{\hat
f,t,\lfloor \log r_t\rfloor}\stackrel{(\ref{ctftlrt1})}{=}1$).

We set $T_{t,l}=\{\Omega_{{\cal A}'}\}_{{\cal A}'\in
\mathbf{T}_{t,l}}$, $$\hat T_{t,l}=\{E\cap E':\; E\in T_{t,l},
\;E'\in T_{t,l+1}, \; {\rm mes}\, (E\cap E')>0\}.$$

By construction, for any function $f\in \hat X_p(\Omega)$ we have
$\hat P_{\mathbf{T}_{t,l}}f-\hat P_{\mathbf{T}_{t,l+1}}f
\stackrel{(\ref{hptfsto})}{\in} {\cal S}_{\hat{T}_{t,l}}(\Omega)$.
Let $s'_{t,l}=\dim \, {\cal S}_{\hat{T}_{t,l}}(\Omega)$. From
(\ref{ctftl}) and (\ref{ctinter}) it follows that there exists
$C(\mathfrak{Z}_0)\ge 1$ such that
\begin{align}
\label{sstl} s'_{t,l}\le \lceil C(\mathfrak{Z}_0)r_t\cdot
2^{-l}\rceil =:s''_{t,l}.
\end{align}

By Lemma \ref{oper_a}, there exists an isomorphism
$\overline{A}_{t,l}:{\cal S}_{\hat T_{t,l}}(\Omega) \rightarrow
\R^{s'_{t,l}}$ such that
\begin{align}
\|\overline{A}_{t,l}\|_{Y_{p,q,\hat T_{t,l}}(G_t)\rightarrow
l_p^{s'_{t,l}}} \underset{\mathfrak{Z}_0}{\lesssim} 1, \quad
\|\overline{A}_{t,l} ^{\,-1}\| _{l_q^{s'_{t,l}} \rightarrow
Y_q(G_t)}\underset{\mathfrak{Z}_0}{\lesssim} 1.
\end{align}
Hence, by (\ref{mult_n}) it suffices to estimate the sum
$$
\sum \limits _{t=t_*(n)}^{t_{**}(n)-1} \sum \limits _{0\le l\le
\log r_t-1}\|\hat P_{\mathbf{T}_{t,l}}-\hat
P_{\mathbf{T}_{t,l+1}}\|_{\hat X_p(\Omega) \rightarrow Y_{p,q,
\hat T_{t,l}}(\Omega)} e_{\hat k_{t,l}} (I_{s''_{t,l}}:
l_p^{s''_{t,l}} \rightarrow l_q^{s''_{t,l}})=:S.
$$

Let us estimate $\|\hat P_{\mathbf{T}_{t,l}}-\hat
P_{\mathbf{T}_{t,l+1}}\|_{\hat X_p(\Omega) \rightarrow Y_{p,q,
\hat T_{t,l}}(\Omega)}$. Consider a function $f\in B\hat
X_p(\Omega)$. Then
$$
\|\hat P_{\mathbf{T}_{t,l}}f-\hat
P_{\mathbf{T}_{t,l+1}}f\|_{p,q,\hat T_{t,l}} =\left(\sum \limits
_{E\in \hat T_{t,l}} \|\hat P_{\mathbf{T}_{t,l}}f-\hat
P_{\mathbf{T}_{t,l+1}}f\|^p_{Y_q(E)}\right)^{1/p}
\stackrel{(\ref{ctinter})}{\underset{\mathfrak{Z}_0}{\lesssim}}
$$
$$
\left(\sum \limits _{E'\in T_{t,l}} \|P_{t,m_t}f-Q_tf- \hat
P_{\mathbf{T}_{t,l}}f\| _{Y_q(E')}^p +\sum \limits _{E''\in
T_{t,l+1}} \|P_{t,m_t}f-Q_tf- \hat P_{\mathbf{T}_{t,l+1}}f\|
_{Y_q(E'')}^p\right)^{1/p}
$$
$$
\stackrel{(\ref{ptmt1}),(\ref{ptmt2}),(\ref{ptmt3})}
{\underset{\mathfrak{Z}_0}{\lesssim}} 2^{-\lambda_*k_*t}
u_*(2^{k_*t});
$$
i.e.,
\begin{align}
\label{norm_ptl} \|\hat P_{\mathbf{T}_{t,l}}-\hat
P_{\mathbf{T}_{t,l+1}}\|_{\hat X_p(\Omega) \rightarrow Y_{p,q,
\hat T_{t,l}}(\Omega)} \underset{\mathfrak{Z}_0}{\lesssim}
2^{-\lambda_*k_*t} u_*(2^{k_*t}).
\end{align}

Since for $0\le l\le \log r_t$ and small $\varepsilon>0$
$$
\frac{s''_{t,l}}{\hat k_{t,l}}
\stackrel{(\ref{ktl}),(\ref{sstl})}{\underset{\mathfrak{Z}_0}{\asymp}}
\frac{r_t\cdot 2^{-l}}{\lceil n\cdot
2^{-\varepsilon(|t-t_2(n)|+l)}\rceil}
\stackrel{(\ref{rt_n}),(\ref{rt_n1})}{\underset{\mathfrak{Z}_0}{\lesssim}}
1,
$$
by Theorem \ref{shutt_trm} we get $$e_{\hat k_{t,l}}
(I_{s''_{t,l}}: l_p^{s''_{t,l}} \rightarrow l_q^{s''_{t,l}})
\underset{\mathfrak{Z}_0}{\lesssim} (s''_{t,l})^{\frac 1q-\frac
1p} \cdot 2^{-\frac{\hat k_{t,l}}{s''_{t,l}}}.$$ Moreover, there
exists $m_*=m_*(\mathfrak{Z}_0)$ such that for $0\le \nu<m_*$ the
sequence $\left\{\frac{\hat
k_{t,m_*l+\nu}}{s''_{t,m_*l+\nu}}\right\}_{0\le m_*l+\nu\le \log
r_t}$ increases not slower than some geometric progression. This
together with (\ref{norm_ptl}) yields that
$$
S \underset{\mathfrak{Z}_0}{\lesssim} \sum \limits
_{t=t_*(n)}^{t_{**}(n)-1}2^{-\lambda_*k_*t} u_*(2^{k_*t}) \max
_{0\le \nu<m_*}(s''_{t,\nu})^{\frac 1q-\frac 1p} \cdot
2^{-\frac{\hat
k_{t,\nu}}{s''_{t,\nu}}}\underset{\mathfrak{Z}_0}{\lesssim}
$$
$$
\lesssim \sum \limits _{t=t_*(n)}^{t_{**}(n)-1}2^{-\lambda_*k_*t}
u_*(2^{k_*t}) (\hat k_{t,0})^{\frac 1q-\frac 1p}=:S'.
$$

If (\ref{nu_t_k}) holds, then
$$
S' \stackrel{(\ref{ktl})}{\underset{\mathfrak{Z}_0}{\lesssim}}
2^{-\lambda_*k_*t_*(n)} u_*(2^{k_*t_*(n)})(\hat
k_{t_*(n),0})^{\frac 1q-\frac 1p}\stackrel{(\ref{nu_t_est1}),
(\ref{t2n})}{\underset{\mathfrak{Z}_0}{\lesssim}}
n^{-\lambda_*\beta_*+\frac 1q-\frac 1p} \varphi_*^{-\lambda_*}(n)
u_*(n^{\beta_*} \varphi_*(n)).
$$
Hence,
\begin{align}
\label{slpt1} \sum \limits _{t=t_*(n)}^{t_{**}(n)-1} \sum \limits
_{0\le l\le \log r_t-1} e_{\hat k_{t,l}} (\hat
P_{\mathbf{T}_{t,l}}-\hat P_{\mathbf{T}_{t,l+1}}:\hat X_p(\Omega)
\rightarrow
Y_q(G_t))\underset{\mathfrak{Z}_0}{\lesssim}n^{-\lambda_*\beta_*+\frac
1q-\frac 1p} \varphi_*^{-\lambda_*}(n) u_*(n^{\beta_*}
\varphi_*(n)).
\end{align}

Let (\ref{nu_t_k1}) hold. Then for $\lambda_*>\frac 1p-\frac 1q$
$$
S'\stackrel{(\ref{ktl})}{\underset{\mathfrak{Z}_0}{\lesssim}}
2^{-\lambda_*k_*t_*(n)} u_*(2^{k_*t_*(n)})(\hat
k_{t_*(n),0})^{\frac 1q-\frac 1p}\stackrel{(\ref{nu_t_est2}),
(\ref{t2n})}{\underset{\mathfrak{Z}_0}{\lesssim}} n^{\frac
1q-\frac 1p} (\log n)^{-\lambda_*} u_*(\log n).
$$
Hence,
\begin{align}
\label{slpt2} \sum \limits _{t=t_*(n)}^{t_{**}(n)-1} \sum \limits
_{0\le l\le \log r_t-1} e_{\hat k_{t,l}} (\hat
P_{\mathbf{T}_{t,l}}-\hat P_{\mathbf{T}_{t,l+1}}:\hat X_p(\Omega)
\rightarrow Y_q(G_t))\underset{\mathfrak{Z}_0}{\lesssim} n^{\frac
1q-\frac 1p} (\log n)^{-\lambda_*+\frac 1p-\frac 1q} u_*(\log n).
\end{align}
Let $\lambda_*<\frac 1p-\frac 1q$. Then $\hat k_{t,0}
\stackrel{(\ref{ktl})}{\ge} n\cdot 2^{-\varepsilon t_{**}(n)}
\stackrel{(\ref{nu_t_est2})}{\underset{\mathfrak{Z}_0}{\asymp}}
n^{1-\varepsilon}$. Therefore, for sufficiently small
$\varepsilon>0$
$$
S' \underset{\mathfrak{Z}_0}{\lesssim} 2^{-\lambda_*k_*t_*(n)}
u_*(2^{k_*t_*(n)})(\hat k_{t_*(n),0})^{\frac 1q-\frac 1p}
\stackrel{(\ref{nu_t_est2})}{\underset{\mathfrak{Z}_0}{\lesssim}}
$$
$$
\lesssim n^{\left(\frac 1q-\frac 1p\right)(1-\varepsilon)} (\log
n)^{-\lambda_*} u_*(\log n)\underset{\mathfrak{Z}_0}{\lesssim}
n^{-\lambda_*} u_*(n);
$$
i.e.,
\begin{align}
\label{slpt3} \sum \limits _{t=t_*(n)}^{t_{**}(n)-1} \sum \limits
_{0\le l\le \log r_t-1} e_{\hat k_{t,l}} (\hat
P_{\mathbf{T}_{t,l}}-\hat P_{\mathbf{T}_{t,l+1}}:\hat X_p(\Omega)
\rightarrow Y_q(G_t))\underset{\mathfrak{Z}_0}{\lesssim}
n^{-\lambda_*} u_*(n).
\end{align}

{\bf Step 7.} Let ${\cal N}'_t=\{\mathbf{T}_{\hat
f,t,0}|_{\Gamma_t}:\; \hat f\in B\hat X_p(\Omega)\}$. Then ${\cal
N}'_t \stackrel{(\ref{tft0nrt})}{\subset} {\cal N}_{r_t-1}$, and
\begin{align}
\label{log_nsrt} \log |{\cal N}'_t| \stackrel{(\ref{t2n}),
(\ref{log_nt})} {\underset {\mathfrak{Z}_0,\varepsilon}{\lesssim}}
n\cdot 2^{-\varepsilon|t-t_2(n)|} \quad \text{if}\quad
(\ref{nu_t_k}) \text{ holds},
\end{align}
\begin{align}
\label{log_nsrt1} \log |{\cal N}'_t| \stackrel{(\ref{log_nt1})}
{\underset {\mathfrak{Z}_0,\varepsilon}{\lesssim}} \max \{n\cdot
2^{-\varepsilon|t-t_2(n)|}, \, 2^t\} \quad \text{if}\quad
(\ref{nu_t_k1})\text{ holds}.
\end{align}
Let $\hat k_{t,l}$ be defined by (\ref{ktl}). We set
\begin{align}
\label{hkt} \hat k_t=\sum \limits _{0\le l\le \log r_t-1} (\hat
k_{t,l}-1)+1 \underset{\mathfrak{Z}_0,\varepsilon}{\lesssim}
n\cdot 2^{-\varepsilon|t-t_2(n)|}.
\end{align}

By Theorem \ref{lifs_sta},
\begin{align}
\label{ehkt}
\begin{array}{c} \displaystyle
e_{\hat k_t+[\log |{\cal N}'_t|]+1} (P_{t,m_t}-Q_t:\hat
X_p(\Omega) \rightarrow Y_q(G_t)) \le \\
\le \sup _{{\bf T} \in {\cal N}'_t} e_{\hat k_t}(\hat
P_{\mathbf{T}}:\hat X_p(\Omega) \rightarrow Y_q(G_t))+\\+
\sup_{f\in B\hat X_p(\Omega)} \inf _{\mathbf{T}\in {\cal N}'_t}
\|P_{t,m_t}f-Q_tf-\hat P_{\mathbf{T}}f\|_{Y_q(G_t)}\le
\\
\le \sup _{\hat f\in B\hat X_p(\Omega)} e_{\hat k_t}(\hat
P_{\mathbf{T}_{\hat f,t,0}}:\hat X_p(\Omega) \rightarrow
Y_q(G_t))+ \\ +\sup_{f\in B\hat X_p(\Omega)}
\|P_{t,m_t}f-Q_tf-\hat P_{\mathbf{T}_{f,t,0}}f\|_{Y_q(G_t)}.
\end{array}
\end{align}
Let $\tilde k_t=\tilde k_t(n)=\hat k_t+[\log |{\cal N}'_t|]+1$.
Then
$$
\tilde k_t-1\stackrel{(\ref{log_nsrt}),(\ref{hkt})}
{\underset{\mathfrak{Z}_0,\varepsilon}{\lesssim}} n\cdot
2^{-\varepsilon|t-t_2(n)|} \quad \text{for}\quad (\ref{nu_t_k}),
$$
$$
\tilde k_t-1\stackrel{(\ref{log_nsrt1}),(\ref{hkt})}
{\underset{\mathfrak{Z}_0,\varepsilon}{\lesssim}} \max \{n\cdot
2^{-\varepsilon|t-t_2(n)|}, \, 2^t\} \quad \text{for}\quad
(\ref{nu_t_k1});
$$
i.e., (\ref{til_kt_est}), (\ref{til_kt_est1}) hold. Further,
$$
\sum \limits_{t=t_*(n)}^{t_{**}(n)-1} e_{\tilde k_t}
(P_{t,m_t}-Q_t:\hat X_p(\Omega) \rightarrow Y_q(G_t))
\stackrel{(\ref{ehkt})}{\le}
$$
$$
\le \sum \limits_{t=t_*(n)}^{t_{**}(n)-1} \sup _{\hat f\in B\hat
X_p(\Omega)} e_{\hat k_t}(\hat P_{\mathbf{T}_{\hat f,t,0}}:\hat
X_p(\Omega) \rightarrow Y_q(G_t))+
$$
$$
+\sum \limits_{t=t_*(n)}^{t_{**}(n)-1} \sup_{f\in B\hat
X_p(\Omega)} \|P_{t,m_t}f-Q_tf-\hat
P_{\mathbf{T}_{f,t,0}}f\|_{Y_q(G_t)}.
$$
We apply (\ref{eklspt}) and (\ref{slpt1}), (\ref{slpt2}),
(\ref{slpt3}) to estimate the first summand, and we use
(\ref{at}), (\ref{at_sum}), (\ref{at_sum1}) and (\ref{at_sum2}) to
estimate the second summand.

This completes the proof.
\end{proof}
The relations (\ref{eklspt}), (\ref{f_q_t}), (\ref{f_p_t_m}) and
Lemmas \ref{qt_st_n}, \ref{sum_qt_est}, \ref{ptm_sum},
\ref{t_g_tstn}, \ref{entr_pmtqt} yield Theorems \ref{trm1} and
\ref{trm2} for $p<q$.

\renewcommand{\proofname}{\bf Proof}

\begin{Rem}
\label{diskr_case} Suppose that Assumptions \ref{sup1} and
\ref{sup3} hold, and Assumption \ref{sup2} is substituted by the
following condition: for any $\xi\in {\bf V}({\cal A})$ the set
$F(\xi)$ is the atom of ${\rm mes}$. Then the assertions of
Theorems \ref{trm1} and \ref{trm2} hold with $\delta_*=+\infty$.
\end{Rem}

\section{Estimates of entropy numbers of weighted Sobolev spaces}

Let $\Omega\in {\bf FC}(a)$, and let $\Gamma \subset \partial
\Omega$ be an $h$-set. Suppose that there exists $c_0\ge c_*$ such
that
\begin{align}
\label{hths} \frac{h(t)}{h(s)}\le c_0, \quad j\in {\mathbb{N}},
\quad t, \quad s\in [2^{-j-1}, \, 2^{-j+1}]
\end{align}
(here $c_*$ is the constant from Definition \ref{h_set}). In
\cite{vas_vl_raspr, vas_vl_raspr2} there were constructed a tree
${\cal A}$ with vertex set $\{\eta_{j,i}\}_{j\ge j_{\min}, i\in
\tilde I_j}$, a number $\overline{s}=\overline{s}(a, \, d)\in \N$
and a partition $\{\Omega [\eta_{j,i}]\} _{j\ge j_{\min}, i\in
\tilde I_j}$ of the domain $\Omega$ with following properties:

\begin{enumerate}
\item $\eta_{j_{\min},1}$ is the minimal vertex in ${\cal A}$, and for any
$j\ge j_{\min}$ the family of sets $\{{\bf V}_1^{\cal
A}(\eta_{j,i})\}_{i\in \tilde I_j}$ form the partition of
$\{\eta_{j+1,t}\}_{t\in \tilde I_{j+1}}$.
\item For any $j\ge j_{\min}$, $i\in \tilde I_j$ we have $\Omega
[\eta_{j,i}] \in {\bf FC}(b_*)$ with $b_*=b_*(a, \, d)>0$.
\item ${\rm diam}\, \Omega[\eta_{j,i}] \underset{a,d,c_0}{\asymp}
2^{-\overline{s}j}$.
\item For any $x\in \Omega[\eta_{j,i}]$ we have ${\rm dist}\, (x, \, \Gamma)
\underset{a,d,c_0}{\asymp} 2^{-\overline{s}j}$.
\item For any $j\ge j_{\min}$, $i\in \tilde I_j$, $j'\ge j$
we have ${\rm card}\, {\bf V}_{j'-j}^{{\cal A}}(\eta_{j,i})
\underset{a,d,c_0}{\lesssim} \frac{h(2^{-\overline{s}j})}
{h(2^{-\overline{s}j'})}$; in particular,
\begin{align}
\label{ctij} {\rm card}\, \tilde I_j \underset{a,d,c_0}{\lesssim}
\frac{h(2^{-\overline{s}j_{\min}})} {h(2^{-\overline{s}j})}.
\end{align}
\end{enumerate}

Suppose that conditions of Theorem \ref{th1}, \ref{th2} or
\ref{th3} hold (then we have (\ref{hths}) with
$c_0=c_0(\mathfrak{Z})$). We define weight functions $u$, $w: {\bf
V}({\cal A}) \rightarrow (0, \, \infty)$ as follows:
\begin{align}
\label{ueta_ji} u(\eta_{j,i})=u_j=g(2^{-\overline{s}j})\cdot
2^{-\left(r-\frac dp\right)\overline{s}j}, \quad w(\eta_{j,i})
=w_j = v(2^{-\overline{s}j})\cdot 2^{-\frac{d}{q}\overline{s}j}.
\end{align}
For each subtree ${\cal D}\subset {\cal A}$ we denote
$\Omega[{\cal D}]=\cup _{\xi \in {\bf V}({\cal D})} \Omega[\xi]$.
It was proved in \cite{vas_vl_raspr2, vas_sib, vas_w_lim} that for
any $j_0\ge j_{\min}$, $i_0\in \tilde I_{j_0}$ and for any vertex
$\eta_{j_0,i_0}$ there exists a linear continuous operator
$P_{\eta_{j_0,i_0}}:L_{q,v}(\Omega) \rightarrow {\cal
P}_{r-1}(\Omega)$ such that for any subtree ${\cal D}$ with
minimal vertex $\eta_{j_0,i_0}$ and for any function $f\in
W^r_{p,g}(\Omega)$
\begin{align}
\label{fpeta} \|f-P_{\eta_{j_0,i_0}}f\|_{L_{q,v}(\Omega[{\cal
D}])} \underset{\mathfrak{Z}}{\lesssim} C(j_0)
\left\|\frac{\nabla^r f}{g}\right\|_{L_p(\Omega[{\cal D}])};
\end{align}
here $C(j_0)$ is defined as follows.
\begin{enumerate}
\item \label{bvlq} Let $\beta_v<\frac{d-\theta}{q}$. Then (see
\cite{vas_vl_raspr2})
$$
C(j_0)=\sup _{j\ge j_0} u_jw_j \quad\text{for} \quad p\le q,$$$$
C(j_0)=\left(\sum \limits _{j\ge j_0} (u_jw_j)^{\frac{pq}{p-q}}
\frac{h(2^{-\overline{s}j_0})}{h(2^{-\overline{s}j})}
\right)^{\frac 1q-\frac 1p} \quad\text{for} \quad p> q.
$$
\item Let $\theta>0$, $\beta_v=\frac{d-\theta}{q}$. Then (see \cite{vas_w_lim})
\begin{align}
\label{cj0db} \begin{array}{c}
C(j_0)=2^{-(\delta-\beta)\overline{s}j_0}
(\overline{s}j_0)^{-\alpha+\frac 1q} \rho(\overline{s}j_0),\quad \text{if}\quad p<q\\
\text{or} \quad p\ge q, \quad \beta-\delta<-\theta\left(\frac
1q-\frac 1p\right), \end{array}
\end{align}
\begin{align}
\label{cj0db1} C(j_0)=2^{-\theta\left(\frac 1q-\frac
1p\right)\overline{s}j_0} (\overline{s}j_0)^{-\alpha+1+\frac
1q-\frac 1p} \rho(\overline{s}j_0) \quad \text{for} \; p\ge q, \;
\beta-\delta=-\theta\left(\frac 1q-\frac 1p\right).
\end{align}
\end{enumerate}

The magnitude $C(j_0)$ in Case \ref{bvlq} is estimated as follows.
If $\beta-\delta<-\theta\left(\frac 1q-\frac 1p\right)_+$, then
\begin{align}
\label{cj01} C(j_0) \underset{\mathfrak{Z}}{\lesssim}
2^{(\beta-\delta) \overline{s}j_0} (\overline{s}j_0)^{-\alpha}
\rho(\overline{s}j_0),
\end{align}
if $\beta-\delta=-\theta\left(\frac 1q-\frac 1p\right)_+$ and
$\alpha>(1-\gamma)\left(\frac 1q-\frac 1p\right)_+$, then
\begin{align}
\label{cj02thalp} C(j_0) \underset{\mathfrak{Z}}{\lesssim}
2^{-\theta\left(\frac 1q-\frac 1p\right)_+\overline{s}j_0}
j_0^{-\alpha+\left(\frac 1q-\frac 1p\right)_+} \rho(j_0)
\end{align}
(these estimates are proved in \cite{vas_vl_raspr2}). Suppose that
conditions of assertion 2 of Theorem \ref{th3} hold: i.e.,
$\theta=0$,
\begin{align}
\label{a1g} \beta-\delta = 0, \quad \alpha=(1-\gamma)\left(\frac
1q-\frac 1p\right)_+, \quad \lambda> (1-\nu) \left(\frac 1q-\frac
1p\right)_+.
\end{align}
Then
\begin{align}
\label{cj0l} C(j_0) \stackrel{(\ref{ghi_g0}),
(\ref{ueta_ji})}{\underset{\mathfrak{Z}}{\lesssim}} (\log
(\overline{s}j_0))^{-\lambda}, \quad p\le q,
\end{align}
\begin{align}
\label{cj0ll}
\begin{array}{c}
C(j_0) \stackrel{(\ref{def_h}), (\ref{ghi_g0}), (\ref{ueta_ji}),
(\ref{a1g})}{\underset{\mathfrak{Z}}{\lesssim}} (\overline{s}j_0)
^{\gamma\left(\frac 1q-\frac 1p\right)} [\log
(\overline{s}j_0)]^{\nu \left(\frac 1q-\frac 1p\right)} \left(
\sum \limits _{j\ge j_0} (\overline{s}j)^{-1} (\log
(\overline{s}j))^{-\lambda \frac{pq}{p-q}-\nu}\right)^{\frac
1q-\frac 1p}\underset{\mathfrak{Z}}{\lesssim}
\\
\underset{\mathfrak{Z}}{\lesssim}  (\overline{s}j_0)
^{\gamma\left(\frac 1q-\frac 1p\right)} [\log (\overline{s}j_0)]
^{-\lambda+\frac 1q-\frac 1p}, \quad p>q.
\end{array}
\end{align}

In proofs of lower estimates of entropy numbers we use the
following assertions.

\begin{Lem}
\label{low} Let $\Omega\subset \R^d$ be a domain, let
$W^r_{p,g}(\Omega) \subset L_{q,v}(\Omega)$, let $G_1, \, \dots,
\, G_m\subset \Omega$ be pairwise non-overlapping sets, and let
$\psi_1, \, \dots, \, \psi_m\in W^r_{p,g}(\Omega)$,
$\left\|\frac{\nabla^r\psi_j}{g}\right\|_{L_p(\Omega)}=1$, ${\rm
supp} \, \psi_j\subset G_j$,
\begin{align}
\label{psi_j_yq_ge_m} \|\psi_j\|_{L_{q,v}(G_j)}\ge M,\quad 1\le
j\le m.
\end{align}
Let $X={\rm span}\, \{\psi_j\}_{j=1}^m$ be equipped with norm
$\|f\|_X=\left\|\frac{\nabla^r f}{g}\right\|_{L_p(\Omega)}$, and
let ${\rm id}:X\rightarrow L_{q,v}(\Omega)$ be the embedding
operator. Then for any $n\in \{0, \, \dots, \, m\}$
$$
e_n({\rm id}:X\rightarrow L_{q,v}(\Omega)) \ge M\cdot
e_n(I_m:l^m_p\rightarrow l^m_q).
$$
In particular, if $X \subset \hat {\cal W}^r_{p,g}(\Omega)$, then
$$
e_n(I:\hat {\cal W}^r_{p,g}(\Omega)\rightarrow L_{q,v}(\Omega))
\ge M\cdot e_n(I_m:l^m_p\rightarrow l^m_q).
$$
\end{Lem}
This lemma is proved similarly as the lower estimate of $n$-widths
in \cite{peak_lim}. Similarly from Theorem \ref{kuhn_trm} we
obtain
\begin{Cor}
\label{low_cor} Let $\Omega\subset \R^d$ be a domain, let
$W^r_{p,g}(\Omega) \subset L_{q,v}(\Omega)$, let $G_j\subset
\Omega$, $j\in \N$, be pairwise non-overlapping sets, and let
$\psi_j\in W^r_{p,g}(\Omega)$,
$\left\|\frac{\nabla^r\psi_j}{g}\right\|_{L_p(\Omega)}=1$, ${\rm
supp} \, \psi_j\subset G_j$,
\begin{align}
\label{psi_j_yq_ge_m1} \|\psi_j\|_{L_{q,v}(G_j)}\ge M_j,\quad j\in
\N.
\end{align}
Let $X={\rm span}\, \{\psi_j\}_{j=1}^\infty \cap {\rm span}\,
W^r_{p,g}(\Omega)$ be equipped with norm
$\|f\|_X=\left\|\frac{\nabla^r f}{g}\right\|_{L_p(\Omega)}$, and
let ${\rm id}:X\rightarrow L_{q,v}(\Omega)$ be the embedding
operator. Denote $\omega_n=\left(\sum \limits _{j=n}^\infty
M_j^{\frac{pq}{p-q}}\right)^{\frac 1q-\frac 1p}$. Suppose that
there exists $C\ge 1$ such that $\omega_n \le C \omega_{2n}$ for
any $n\in \N$. Then for any $n\in \N$ we have
$$
e_n({\rm id}:X\rightarrow L_{q,v}(\Omega)) \underset{p,q,C}
{\gtrsim} \omega_n.
$$
In particular, if $X \subset \hat {\cal W}^r_{p,g}(\Omega)$, then
$$
e_n(I:\hat {\cal W}^r_{p,g}(\Omega)\rightarrow L_{q,v}(\Omega))
\underset{p,q,C} {\gtrsim} \omega_n.
$$
\end{Cor}

\renewcommand{\proofname}{\bf Proof of Theorems \ref{th1}, \ref{th2}, \ref{th3}}
\begin{proof}
{\it The upper estimate.} We set $\hat \Theta
=\{\Omega[\eta_{j,i}]\}_{j\ge j_{\min},i\in \tilde I_j}$, $\hat
F(\eta_{j,i}) =\Omega[\eta_{j,i}]$, $X_p(\Omega) ={\rm span}\,
W^r_{p,g}(\Omega)$, $\hat X_p(\Omega) =\hat {\cal
W}^r_{p,g}(\Omega)$, $Y_q(\Omega)=L_{q,v}(\Omega)$, ${\cal
P}(\Omega) ={\cal P}_{r-1}(\Omega)$.

Similarly as in \cite[p. 49]{vas_width_raspr} we can prove that
Assumption 2 holds with $\delta_*=\frac{\delta}{d}$ and $\tilde
w_*(\eta_{j,i}) \underset{\mathfrak{Z}}{\lesssim} u_jw_j$.

Consider the following cases.
\begin{enumerate}
\item Suppose that one of the following conditions holds:
\begin{itemize}
\item $\beta -\delta< -\theta \left(\frac 1q-\frac 1p\right)_+$,
\item $\beta -\delta= -\theta\left(\frac 1q-\frac 1p\right)_+$, $p\ge
q$,
\item $\theta=0$, $\beta-\delta=0$, $\alpha>(1-\gamma)\left(\frac
1q -\frac 1p\right)_+$.
\end{itemize}
Then the partition $\{{\cal A}_{t,i}\}_{t\ge t_0,i\in \hat J_t}$
is constructed similarly as in \cite[p. 49--50]{vas_width_raspr},
and by properties 1--5 of the tree ${\cal A}$ and (\ref{fpeta}),
(\ref{cj0db}), (\ref{cj0db1}), (\ref{cj01}), (\ref{cj02thalp}) we
get that Assumptions 1 and 3 hold. Here $\lambda_*$, $\gamma_*$
and $\psi_*$ are the same as in \cite{vas_width_raspr} (see Cases
1, 3, 4, and Case 2 for $p\ge q$; in Cases 1 and 4 we take
$\Lambda(x) =|\log x|^{\gamma} \tau(|\log x|)$). If
$\beta_v<\frac{d-\theta}{q}$, then the function $u_*$ is the same
as in \cite{vas_width_raspr} (in Cases 1 and 4 we take
$\Psi(x)=|\log x|^{-\alpha} \rho(|\log x|)$). If
$\beta_v=\frac{d-\theta}{q}$, then $u_*(y)=(\log y)^{-\alpha+\frac
1q}\rho(\log y)$ for $p<q$ or $\beta-\delta<-\theta\left(\frac
1q-\frac 1p\right)_+$, and $u_*(y)=(\log y)^{-\alpha+1+\frac
1q-\frac 1p} \rho(\log y)$ for $\beta-\delta=-\theta\left(\frac
1q-\frac 1p\right)_+$ and $p\ge q$ (recall that for
$\beta_v=\frac{d-\theta}{q}$ we consider only the case
$\theta>0$).

Applying Theorem \ref{trm1} and Lemma \ref{log}, we get the upper
estimate in assertions 1 and 2a of Theorem \ref{th1}, in Theorem
\ref{th2} and in assertion 1 of Theorem \ref{th3}.

\item Let $\theta>0$, $\beta-\delta=-\theta\left(\frac
1q-\frac 1p\right)_+$, $p<q$. Denote by $\Gamma_t$ the maximal
subgraph in ${\cal A}$ on the vertex set
$$
\{\eta_{j,l}:\; 2^{t-1}<\overline{s}j\le 2^t, \quad l\in \tilde
I_j\},
$$
and by $\{{\cal A}_{t,i}\}_{i\in \hat J_t}$ we denote the set of
connected components of $\Gamma_t$. We set $t_0=\min \{t\in
\Z_+:\; {\bf V}(\Gamma_t)\ne \varnothing\}$. By (\ref{def_h}) and
(\ref{ctij}), ${\rm card}\, {\bf V}(\Gamma_t)
\underset{\mathfrak{Z}_0}{\lesssim} 2^{\theta \cdot 2^t}
2^{-\gamma t} \tau^{-1}(2^t)$. This together with (\ref{fpeta}),
(\ref{cj0db}) and (\ref{cj02thalp}) implies that Assumptions 1 and
3 hold with $\lambda_*=-\alpha_0$ (see assertion 2b of Theorem
\ref{th1}), $u_*(y)=\rho(y)$, $\gamma_*=\theta$, $\psi_*(y)=(\log
y)^{-\gamma} \tau^{-1}(\log y)$ in (\ref{nu_t_k1}). Applying
Theorem \ref{trm2}, we get the upper estimate in assertion 2b of
Theorem \ref{th1}.

\item Let $\theta=0$, $\beta-\delta=0$,
$\alpha=(1-\gamma)\left(\frac 1q-\frac 1p\right)_+$.

\begin{enumerate}
\item Suppose that $p\ge q$. We define the partition
$\{{\cal A}_{t,i}\}_{t\ge t_0,i\in \hat J_t}$ similarly as in
\cite[p. 50]{vas_width_raspr} (see Case 3). Then Assumptions 1 and
3 hold with $\gamma_*=1-\gamma$, $\lambda_*=\alpha+\frac 1p-\frac
1q=-\gamma\left(\frac 1q-\frac 1p\right)$,
$\mu_*=\alpha=(1-\gamma)\left(\frac 1q-\frac 1p\right)$,
$\psi_*(x)=(\log x)^{-\nu}$, $u_*(x)=(\log x)^{-\lambda+\frac
1q-\frac 1p}$ (see (\ref{cj0l}), (\ref{cj0ll})). Hence,
$\beta_*=\frac{1}{1-\gamma}$, $\varphi_*(x)=(\log
x)^{\frac{\nu}{1-\gamma}}$ (see Lemma \ref{log}). In addition,
${\rm card}\, \hat J_t
\stackrel{(\ref{ctij})}{\underset{\mathfrak{Z}}{\lesssim}}
2^{-\gamma t} t^{-\nu}$. The relations (\ref{bipf4684gn}) and
(\ref{2ll}) follow from the conditions
$\alpha=(1-\gamma)\left(\frac 1q-\frac 1p\right)$ and
$\lambda>(1-\nu) \left(\frac 1q-\frac 1p\right)$. Let us check
that (\ref{2l}) holds. Indeed, $2^{-\lambda_*k_*t}
u_*(2^{k_*t})=2^{\gamma\left(\frac 1q-\frac 1p\right)t}
t^{-\lambda+\frac 1q-\frac 1p}$. Since the function $h$ is
non-decreasing, we have $\gamma\le 0$; moreover, $\nu\le 0$ for
$\gamma=0$. If $\gamma<0$ and $p>q$, then (\ref{2l}) follows from
the inequality $\gamma\left(\frac 1q-\frac 1p\right)<0$. If
$\gamma=0$ and $p>q$, then (\ref{2l}) follows from the inequality
$-\lambda+\frac 1q-\frac 1p< \nu\left(\frac 1q-\frac 1p\right)\le
0$. If $p=q$, then the assertion follows from the inequality
$\lambda>0$.

Notice that for $p=q$ we have $\lambda_*=\mu_*$. Since
$\mu_*\beta_*=\frac 1q-\frac 1p<\frac{\delta}{d}=\delta_*$, we get
by Theorem \ref{trm1} that
$$
e_n(I:\hat {\cal W}^r_{p,g}(\Omega) \rightarrow L_{q,v}(\Omega))
\underset{\mathfrak{Z}}{\lesssim} u_*(n^{\beta_*}\varphi_*(n))
\varphi_*^{-\mu_*}(n) \underset{\mathfrak{Z}}{\asymp} (\log
n)^{-\lambda+ \frac 1q-\frac 1p-\nu\left(\frac 1q-\frac
1p\right)}.
$$
Thus, we obtain the upper estimate in assertion 2a of Theorem
\ref{th3}.

\item Let $p<q$. We set ${\bf V}(\Gamma_t)=\{\eta_{j,i}:\;
2^{2^{t-1}}<j\le 2^{2^t}\}$ and denote by $\{{\cal
A}_{t,i}\}_{i\in \hat J_t}$ the set of connected components of the
graph $\Gamma_t$. Then
$$
{\rm card}\, {\bf V}(\Gamma_t) \underset{\mathfrak{Z}}{\lesssim}
\sum \limits _{2^{2^{t-1}}+1\le \overline{s}j\le2^{2^t}}
(\overline{s}j)^{-\gamma} (\log (\overline{s}j))^{-\nu}
\underset{\mathfrak{Z}}{\lesssim} 2^{(1-\gamma)2^t} 2^{-\nu t}
=:\overline{\nu}_t
$$
(i.e., (\ref{nu_t_k1}) holds). This together with (\ref{cj0l})
implies that Assumptions 1 and 3 hold with
$\lambda_*=\mu_*=\lambda$, $u_*\equiv 1$. Applying Theorem
\ref{trm2}, we get the upper estimate in assertion 2b of Theorem
\ref{th3}.
\end{enumerate}
\end{enumerate}

{\it The lower estimate.} Similarly as in \cite[p.
50]{vas_width_raspr} we can prove that
$$
e_n(I:\hat {\cal W}^r_{p,g}(\Omega) \rightarrow L_{q,v}(\Omega))
\underset{\mathfrak{Z}_*}{\gtrsim} e_n(I:\hat {\cal W}^r_p([0, \,
1]^d) \rightarrow L_q([0, \, 1]^d)) \underset{p,q,r,d}{\asymp}
n^{-\frac{\delta}{d} +\frac 1q-\frac 1p}.
$$

This gives the desired lower estimates in Theorem \ref{th2}, in
assertion 1 of Theorem \ref{th1} for $\frac{\delta}{d}<
\frac{\delta-\beta}{\theta}$ and in assertion 1 of Theorem
\ref{th3} for $\frac{\delta}{d}< \frac{\alpha}{1-\gamma}$.

In other cases we apply Lemma \ref{low} or Corollary
\ref{low_cor}.

In \cite[p. 50]{vas_width_raspr}, \cite{vas_vl_raspr2} and
\cite{vas_w_lim} the number $k_{**}=k_{**}(\mathfrak{Z}_*) \in \N$
is defined and the functions $\{\psi_{t,j}\}_{j\in J_t}\in
C^\infty(\R^d)$ are constructed with the following properties:
\begin{align}
\label{cjt_low} {\rm card}\, J_t
\underset{\mathfrak{Z}_*}{\gtrsim} 2^{\theta k_{**} t}
(k_{**}t)^{-\gamma} \tau^{-1}(k_{**}t),
\end{align}
$\left\|\frac{\nabla^r \psi_{t,j}}{g}\right\|_{L_p(\Omega)}=1$.
Moreover, $\|\psi_{t,j}\|_{L_{q,v}(\Omega)}$ is estimated from
below as follows.
\begin{enumerate}
\item If $\beta_v<\frac{d-\theta}{q}$, then
\begin{align}
\label{psitj1}
\|\psi_{t,j}\|_{L_{q,v}(\Omega)}\underset{\mathfrak{Z}_*}{\gtrsim}
2^{k_{**}t(\beta-\delta)} (k_{**}t)^{-\alpha} \rho(k_{**}t).
\end{align}
\item Let $\theta>0$, $\beta_v=\frac{d-\theta}{q}$; in addition, we suppose that $p<q$ or $p\ge
q$, $\beta-\delta<-\theta\left(\frac 1q-\frac 1p\right)_+$. Then
\begin{align}
\label{psitj2}
\|\psi_{t,j}\|_{L_{q,v}(\Omega)}\underset{\mathfrak{Z}_*}{\gtrsim}
2^{k_{**}t(\beta-\delta)} (k_{**}t)^{-\alpha+\frac 1q}
\rho(k_{**}t).
\end{align}
\item If $\theta>0$, $p\ge q$, $\beta_v=\frac{d-\theta}{q}$ and $\beta-\delta=-\theta\left(\frac 1q-\frac
1p\right)$, then
\begin{align}
\label{psitj3}
\|\psi_{t,j}\|_{L_{q,v}(\Omega)}\underset{\mathfrak{Z}_*}{\gtrsim}
2^{-\theta\left(\frac 1q-\frac 1p\right)k_{**}t}
(k_{**}t)^{-\alpha+\frac 1q+1-\frac 1p} \rho(k_{**}t).
\end{align}
\end{enumerate}
In addition, in Case 1 the supports of $\psi_{t,j}$ do not overlap
pairwise for different $(t, \, j)$; in Cases 2, 3 for any $t$ the
supports of $\psi_{t,j}$ do not overlap pairwise for different
$j$.

Moreover, it follows from the construction of functions
$\psi_{t,i}$ that for any $x\in {\rm supp}\, \psi_{t,j}$
$${\rm dist}\, (x, \, \Gamma) \underset{\mathfrak{Z}_*}{\lesssim}
2^{-k_{**}t}.$$ Hence, by Remark \ref{peq0}, there exists $\hat
t=\hat t(\mathfrak{Z}_*)\in \N$ such that for $t\ge \hat t$ we
have $P\psi_{t,i}=0$ and $\psi_{t,i}\in \hat W^r_{p,g}(\Omega)$.

\begin{enumerate}
\item Suppose that $\theta>0$.
\begin{enumerate}
\item Let $\beta-\delta<-\theta\left(\frac 1q-\frac
1p\right)_+$. Let $\alpha_0$ be as defined in assertion 1 of
Theorem \ref{th1}. We take $t_n$ such that
\begin{align}
\label{2tnn} n\le 2^{\theta k_{**}t_n}(k_{**}t_n)^{-\gamma}
\tau^{-1}(k_{**}t_n) \underset{\mathfrak{Z}_*}{\lesssim} n.
\end{align}
Then
\begin{align}
\label{2ksstn} 2^{k_{**}t_n}\underset{\mathfrak{Z}_*}{\asymp}
n^{\frac{1}{\theta}} (\log n)^{\frac{\gamma}{\theta}}
\tau^{\frac{1}{\theta}} (\log n)
\end{align}
(see Lemma \ref{log}) and
\begin{align}
\label{ksstn} k_{**}t_n \underset{\mathfrak{Z}_*}{\asymp} \log n.
\end{align}
This together with Theorem \ref{shutt_trm} and Lemma \ref{low}
yields that
$$
e_n(I:\hat {\cal W}^r_{p,g}(\Omega) \rightarrow L_{q,v}(\Omega))
\stackrel{(\ref{psitj1}),
(\ref{psitj2})}{\underset{\mathfrak{Z}_*}{\gtrsim}} n^{\frac
1q-\frac 1p} \cdot 2^{k_{**}t_n(\beta-\delta)}
(k_{**}t_n)^{-\alpha_0}
\rho(k_{**}t_n)\underset{\mathfrak{Z}_*}{\asymp}
$$
$$
\asymp n^{\frac{\beta-\delta}{\theta}+\frac 1q-\frac 1p} (\log
n)^{\frac{(\beta-\delta)\gamma}{\theta} -\alpha_0} \rho(\log n)
\tau^{\frac{\beta-\delta}{\theta}}(\log n).
$$
This implies the lower estimate in assertion 1 of Theorem
\ref{th1}.
\item Let $\beta-\delta=-\theta\left(\frac 1q-\frac
1p\right)$, $p\ge q$. We take $t_n$ such that (\ref{2tnn}) holds.

If $\beta_v<\frac{d-\theta}{q}$ and $p=q$, then Theorem
\ref{shutt_trm} and Lemma \ref{low} yield that
$$
e_n(I:\hat {\cal W}^r_{p,g}(\Omega) \rightarrow L_{q,v}(\Omega))
\stackrel{(\ref{psitj1}),(\ref{ksstn})}{\underset{\mathfrak{Z}_*}{\gtrsim}}
(\log n)^{-\alpha} \rho(\log n).
$$
Let $\beta_v<\frac{d-\theta}{q}$, $p>q$. We apply Corollary
\ref{low_cor} and get that there exists such
$m_0=m_0(\mathfrak{Z}_*)$ that
$$
e_n(I:\hat {\cal W}^r_{p,g}(\Omega) \rightarrow L_{q,v}(\Omega))
\stackrel{(\ref{psitj1})}{\underset{\mathfrak{Z}_*}{\gtrsim}}
$$
$$
\gtrsim \left(\sum \limits _{t\ge t_n+m_0} \sum \limits _{i\in
J_t} 2^{-k_{**}t\theta} (k_{**}t)^{\frac{pq}{p-q}\alpha} \rho
^{\frac{pq}{p-q}} (k_{**}t) \right)^{\frac 1q-\frac 1p}
\stackrel{(\ref{cjt_low})}{\underset{\mathfrak{Z}_*}{\gtrsim}}
$$
$$
\gtrsim \left(\sum \limits _{t\ge t_n+m_0} (k_{**}t)^{-\gamma}
\tau^{-1}(k_{**}t)(k_{**}t)^{\frac{pq}{p-q}\alpha} \rho
^{\frac{pq}{p-q}} (k_{**}t) \right)^{\frac 1q-\frac 1p}
\underset{\mathfrak{Z}_*}{\gtrsim}
$$
$$
\gtrsim (k_{**}t_n)^{-\alpha +(1-\gamma) \left(\frac 1q-\frac
1p\right)} \rho(k_{**}t_n) \tau^{\frac 1p-\frac 1q}(k_{**}t_n)
\stackrel{(\ref{ksstn})}{\underset{\mathfrak{Z}_*}{\gtrsim}}
$$
$$
\gtrsim (\log n)^{-\alpha +(1-\gamma) \left(\frac 1q-\frac
1p\right)} \rho(\log n) \tau^{\frac 1p-\frac 1q}(\log n).
$$
Let $\beta_v=\frac{d-\theta}{q}$. Then Theorem \ref{shutt_trm} and
Lemma \ref{low} imply that
$$
e_n(I:\hat {\cal W}^r_{p,g}(\Omega) \rightarrow L_{q,v}(\Omega))
\stackrel{(\ref{psitj3})}{\underset{\mathfrak{Z}_*}{\gtrsim}}
2^{-\theta\left(\frac 1q-\frac 1p\right)k_{**}t_n}
(k_{**}t_n)^{-\alpha+\frac 1q+1-\frac 1p} \rho(k_{**}t_n) n^{\frac
1q-\frac 1p} \stackrel{(\ref{2ksstn}),
(\ref{ksstn})}{\underset{\mathfrak{Z}_*}{\gtrsim}}
$$
$$
\gtrsim (\log n)^{-\gamma\left(\frac 1q-\frac 1p\right)
-\alpha+\frac 1q+1-\frac 1p} \rho(\log n) \tau^{\frac 1p-\frac
1q}(\log n).
$$

Thus, we obtain the desired estimate in assertion 2a of Theorem
\ref{th1}.
\item Let $\beta-\delta=0$, $p< q$. We take $t_n$ such that
\begin{align}
\label{2tnn2} n^2\le 2^{\theta k_{**}t_n}(k_{**}t_n)^{-\gamma}
\tau^{-1} (k_{**}t_n) \underset{\mathfrak{Z}_*}{\lesssim} n^2.
\end{align}
Then $k_{**}t_n\underset{\mathfrak{Z}_*}{\asymp} \log n$. Applying
Theorem \ref{shutt_trm} and Lemma \ref{low}, we obtain that
$$
e_n(I:\hat {\cal W}^r_{p,g}(\Omega) \rightarrow L_{q,v}(\Omega))
\stackrel{(\ref{cjt_low}), (\ref{psitj1}), (\ref{psitj2}),
(\ref{2tnn2})}{\underset{\mathfrak{Z}_*}{\gtrsim}}
$$
$$
\gtrsim (k_{**}t_n)^{-\alpha_0} \rho(k_{**}t_n) n^{\frac 1q-\frac
1p} \log ^{\frac 1p-\frac 1q}\left(1+\frac{n^2}{n}\right)
\underset{\mathfrak{Z}_*}{\asymp} n^{\frac 1q-\frac 1p} (\log n)
^{-\alpha_0+\frac 1p-\frac 1q} \rho(\log n).
$$
Now we take $t_n$ such that
$$
2^n\le 2^{\theta k_{**}t_n}(k_{**}t_n)^{-\gamma} \tau^{-1}
(k_{**}t_n) \underset{\mathfrak{Z}_*}{\lesssim} 2^n.
$$
Then $k_{**}t_n\underset{\mathfrak{Z}_*}{\asymp} n$. Applying
Theorem \ref{shutt_trm} and Lemma \ref{low}, we obtain that
$$
e_n(I:\hat {\cal W}^r_{p,g}(\Omega) \rightarrow L_{q,v}(\Omega))
\stackrel{(\ref{psitj1}),
(\ref{psitj2})}{\underset{\mathfrak{Z}_*}{\gtrsim}}
(k_{**}t_n)^{-\alpha_0} \rho(k_{**}t_n)
\underset{\mathfrak{Z}_*}{\asymp} n^{-\alpha_0} \rho(n).
$$
Thus, we get desired estimates in assertion 2b of Theorem
\ref{th1}.
\end{enumerate}
\item Let $\theta=0$ and conditions of Theorem \ref{th3} hold.
Then
$$
\|\psi_{t,j}\|_{L_{q,v}(\Omega)}
\stackrel{(\ref{psitj1})}{\underset{\mathfrak{Z}_*} {\gtrsim}}
(k_{**}t)^{-\alpha} \log ^{-\lambda}(k_{**}t), \quad {\rm card}\,
J_t \stackrel{(\ref{cjt_low})}{\underset{\mathfrak{Z}_*}
{\gtrsim}} (k_{**}t)^{-\gamma} \log ^{-\nu}(k_{**}t).
$$
Hence, for $2^m\le k_{**}t<2^{m+1}$ and sufficiently large $m\in
\N$
\begin{align}
\label{psitj4} \|\psi_{t,j}\|_{L_{q,v}(\Omega)}
\underset{\mathfrak{Z}_*} {\gtrsim} 2^{-\alpha m} m^{-\lambda},
\quad {\rm card}\, \left(\cup _{2^m\le k_*t<2^{m+1}} J_t\right)
\underset{\mathfrak{Z}_*} {\gtrsim} 2^{m(1-\gamma)} m^{-\nu}.
\end{align}
\begin{enumerate}
\item Let $\alpha-(1-\gamma) \left(\frac 1q-\frac
1p\right)_+>0$. We take $m_n\in \N$ such that
\begin{align}
\label{2mnn} n\le 2^{m_n(1-\gamma)} m_n^{-\nu}
\underset{\mathfrak{Z}_*} {\lesssim} n.
\end{align}
Then $2^{m_n} \underset{\mathfrak{Z}_*}{\asymp}
n^{\frac{1}{1-\gamma}} (\log n)^{\frac{\nu}{1-\gamma}}$ (see Lemma
\ref{log}), $m_n \underset{\mathfrak{Z}_*}{\asymp} \log n$.
Applying Theorem \ref{shutt_trm} and Lemma \ref{low}, we get that
$$
e_n(I:\hat {\cal W}^r_{p,g}(\Omega) \rightarrow L_{q,v}(\Omega))
\stackrel{(\ref{psitj4})}{\underset{\mathfrak{Z}_*} {\gtrsim}}
2^{-\alpha m_n} m_n^{-\lambda} n^{\frac 1q-\frac
1p}\underset{\mathfrak{Z}_*}{\asymp}
$$
$$
\asymp n^{-\frac{\alpha}{1-\gamma}+\frac 1q-\frac 1p} (\log n)
^{-\frac{\alpha \nu}{1-\gamma} -\lambda}.
$$
Thus, we obtain the desired estimate in assertion 1 of Theorem
\ref{th3}.
\item Let $p\ge q$, $\alpha=(1-\gamma)\left(\frac 1q-\frac
1p\right)$. We define $m_n$ by (\ref{2mnn}). For $p=q$ we have
$\alpha=0$ and
$$
e_n(I:\hat {\cal W}^r_{p,g}(\Omega) \rightarrow
L_{q,v}(\Omega))\stackrel{(\ref{psitj4}),
(\ref{2mnn})}{\underset{\mathfrak{Z}_*} {\gtrsim}} (\log
n)^{-\lambda}.
$$
If $p>q$, then we apply Corollary \ref{low_cor} and get that there
exists $\hat m=\hat m(\mathfrak{Z}_*)$ such that
$$
e_n(I:\hat {\cal W}^r_{p,g}(\Omega) \rightarrow
L_{q,v}(\Omega))\stackrel{(\ref{psitj4})}{\underset{\mathfrak{Z}_*}
{\gtrsim}}
$$
$$
\gtrsim \left(\sum \limits _{m=m_n+\hat m}^\infty 2^{m(1-\gamma)}
m^{-\nu} \cdot 2^{-m \alpha \frac{pq}{p-q}} m^{-\lambda
\frac{pq}{p-q}}\right)^{\frac 1q-\frac
1p}\underset{\mathfrak{Z}_*}{\asymp}
$$
$$
\asymp m_n^{-\lambda+(1-\nu)\left(\frac 1q -\frac 1p\right)}
\underset{\mathfrak{Z}_*}{\asymp} (\log
n)^{-\lambda+(1-\nu)\left(\frac 1q -\frac 1p\right)}.
$$
Thus, we get the desired estimate in assertion 2a of Theorem
\ref{th3}.
\item Let $p< q$, $\alpha=0$. First we take $m_n$ such that
\begin{align}
\label{21gmn} 2^{(1-\gamma)m_n}m_n^{-\nu}
\underset{\mathfrak{Z}_*}{\asymp} n^2.
\end{align}
Then $m_n\underset{\mathfrak{Z}_*}{\asymp} \log n$. Applying
Theorem \ref{shutt_trm} and Lemma \ref{low}, we get
$$
e_n(I:\hat {\cal W}^r_{p,g}(\Omega) \rightarrow L_{q,v}(\Omega))
\stackrel{(\ref{psitj4}), (\ref{21gmn})}{\underset{\mathfrak{Z}_*}
{\gtrsim}}
$$
$$
\gtrsim m_n^{-\lambda} n^{\frac 1q-\frac 1p} \log ^{\frac1p -\frac
1q} \left(1+\frac{n^2}{n} \right)\underset{\mathfrak{Z}_*}
{\asymp} n^{\frac 1q-\frac 1p} (\log n)^{-\lambda+\frac 1p-\frac
1q}.
$$

Now we take $m_n$ such that $2^{(1-\gamma)m_n}m_n^{-\nu}
\underset{\mathfrak{Z}_*}{\asymp} 2^n$. Then
$m_n\underset{\mathfrak{Z}_*} {\asymp} n$. Applying Theorem
\ref{shutt_trm} and Lemma \ref{low}, we get
$$
e_n(I:\hat {\cal W}^r_{p,g}(\Omega) \rightarrow L_{q,v}(\Omega))
\stackrel{(\ref{psitj4})}{\underset{\mathfrak{Z}_*} {\gtrsim}}
m_n^{-\lambda} \underset{\mathfrak{Z}_*} {\asymp} n^{-\lambda}.
$$
Thus, we obtain the desired estimate in assertion 2b of Theorem
\ref{th3}.
\end{enumerate}
\end{enumerate}
This completes the proof of Theorems 1, 2, 3.
\end{proof}
\renewcommand{\proofname}{\bf Proof}

\section{Estimates of entropy numbers of weighted summation operators on a tree}

Let ${\cal A}$ be a tree, and let $f:{\bf V}({\cal A})\rightarrow
\R$. We set
$$
\|f\|_{l_p({\cal A})}=\left ( \sum \limits _{\xi \in {\bf V}
({\cal A})}|f(\xi)|^p\right )^{1/p}, \quad \text{if} \quad 1\le
p<\infty, \quad \|f\|_{l_\infty({\cal A})}=\sup _{\xi \in {\bf V}
({\cal A})} |f(\xi)|.
$$
Denote by \label{obozn_lpg}$l_p({\cal A})$ the space of functions
$f:{\bf V}({\cal A})\rightarrow \R$ with finite norm
$\|f\|_{l_p({\cal A})}$.

Let $u$, $w:{\bf V}({\cal A})\rightarrow [0, \, \infty)$ be weight
functions.

Define the summation operator $S_{u,w,{\cal A}}$ by
$$
S_{u,w,{\cal A}}f(\xi) = w(\xi)\sum \limits _{\xi'\le \xi}
u(\xi')f(\xi'), \quad \xi \in {\bf V}({\cal A}), \quad f:{\bf
V}({\cal A}) \rightarrow \R.
$$

In papers of Lifshits and Linde \cite{lifs_m, l_l, l_l1} estimates
for entropy numbers of the operator $S_{u,w,{\cal A}}: l_p({\cal
A}) \rightarrow l_\infty({\cal A})$ or its dual were obtained
under some conditions on $u$, $w$. Here we obtain order estimates
for entropy numbers of the operator $S_{u,w,{\cal A}}: l_p({\cal
A}) \rightarrow l_q({\cal A})$ for $1<p\le \infty$, $1\le
q<\infty$.

Let ${\bf V}({\cal A}) =\{\eta_{j,i}:\; j\in \Z_+, \;\; i\in
I_j\}$. In addition, we suppose that the family of sets $\{{\bf
V}_1^{\cal A}(\eta_{j,i})\}_{i\in I_j}$ forms the partition of the
set $\{\eta_{j+1,t}\}_{t\in I_{j+1}}$. Suppose that for some
$c_*\ge 1$, $m_*\in \N$
$$
c_*^{-1}\frac{h(2^{-m_*j})} {h(2^{-m_*(j+l)})}\le {\bf V}^{\cal
A}_l(\eta_{j,i})\le  c_*\frac{h(2^{-m_*j})} {h(2^{-m_*(j+l)})},
\quad j, \; l\in \Z_+.
$$
Here the function $h$ is defined by (\ref{def_h}), (\ref{yty}) in
some neighborhood of zero. Let $u$, $w:{\bf V}({\cal A})
\rightarrow (0, \, \infty)$, $u(\eta_{j,i})=u_j$,
$w(\eta_{j,i})=w_j$, $j\in \Z_+$, $i\in I_j$,
$$
u_j=2^{-\kappa _um_*j}(m_*j+1)^{-\alpha_u} \rho_u(m_*j+1), \quad
w_j=2^{-\kappa _wm_*j}(m_*j+1)^{-\alpha_w} \rho_w(m_*j+1),
$$
where $\rho_u$, $\rho_w:(0, \, \infty) \rightarrow (0, \, \infty)$
are absolutely continuous functions, $\lim \limits _{y\to \infty}
\frac{y\rho'_u(y)}{\rho_u(y)} = \lim \limits _{y\to \infty}
\frac{y\rho'_w(y)}{\rho_w(y)}=0$. Moreover, we suppose that
$1<p\le \infty$, $1\le q<\infty$,
\begin{align}
\label{muck_1} \kappa_w>\frac{\theta}{q} \quad \text{or} \quad
\kappa_w=\frac{\theta}{q}, \quad \alpha_w>\frac{1-\gamma}{q}.
\end{align}

We set $\kappa=\kappa_u+\kappa_w$, $\alpha=\alpha_u+\alpha_w$,
$\rho(y)=\rho_u(y)\rho_w(y)$, $\mathfrak{Z}=(p, \, q, \, u, \, w,
\, h, \, m_*, \, c_*)$.

Let ${\cal D}$ be a subtree in ${\cal A}$. Denote by
$\mathfrak{S}^{p,q}_{{\cal D},u,w}$ the operator norm of
$S_{u,w,{\cal D}}: l_p({\cal D}) \rightarrow l_q({\cal D})$.
Applying the results of \cite{vas_vl_raspr2}, \cite{vas_har_tree},
we get that for $j\ge 2$ and for any $i\in I_j$ we have
$\mathfrak{S}^{p,q}_{{\cal A}_{\eta_{j,i}},u,w}
\underset{\mathfrak{Z}}{\asymp} C(j)$, where $C(j)$ is defined as
follows.
\begin{enumerate}
\item Let $\kappa_w>\frac{\theta}{q}$, $\kappa>\theta\left(\frac 1q-\frac
1p\right)_+$. Then $C(j)=2^{-\kappa m_*j} (m_*j)^{-\alpha}
\rho(m_*j)$.
\item Let $\kappa_w>\frac{\theta}{q}$, $\kappa=\theta\left(\frac 1q-\frac
1p\right)_+$, $\alpha>(1-\gamma)\left(\frac 1q-\frac 1p\right)_+$.
Then $$C(j)=2^{-\theta\left(\frac 1q-\frac 1p\right)_+ m_*j}
(m_*j)^{-\alpha+\left(\frac 1q-\frac 1p\right)_+} \rho(m_*j).$$
\item Let $\theta>0$, $\kappa_w=\frac{\theta}{q}$. Suppose that either $\kappa>
\theta\left(\frac 1q-\frac 1p\right)_+$ or $\kappa=
\theta\left(\frac 1q-\frac 1p\right)_+$, $\alpha>\frac 1q$, $p<q$.
Then $C(j)=2^{-\kappa m_*j} (m_*j)^{-\alpha+\frac 1q} \rho(m_*j)$.
\item Let $\theta>0$, $\kappa_w=\frac{\theta}{q}$, $\kappa=
\theta\left(\frac 1q-\frac 1p\right)_+$, $p\ge q$,
$\alpha>1+(1-\gamma) \left(\frac 1q-\frac 1p\right)_+$. Then
$C(j)=2^{-\theta\left(\frac 1q-\frac 1p\right) m_*j}
(m_*j)^{-\alpha+1+\frac 1q-\frac 1p} \rho(m_*j)$.
\end{enumerate}

Applying \ref{diskr_case}, we get results similar to Theorems
\ref{th1}, \ref{th3}.

\begin{Trm}
\label{dis_t1} Let $\theta>0$.
\begin{enumerate}
\item Suppose that $\kappa>\theta\left(\frac 1q-\frac
1p\right)_+$. We set $\alpha_0=\alpha$ for
$\kappa_w>\frac{\theta}{q}$ and $\alpha_0=\alpha-\frac 1q$ for
$\kappa_w=\frac{\theta}{q}$. Then
$$
e_n(S_{u,w,{\cal A}}:l_p({\cal A}) \rightarrow l_q({\cal A}))
\underset{\mathfrak{Z}}{\asymp} n^{-\frac{\kappa}{\theta}+\frac
1q-\frac 1p}(\log n)^{-\alpha_0-\frac{\kappa\gamma}{\theta}}
\rho(\log n) \tau^{-\frac{\kappa}{\theta}}(\log n).
$$
\item Suppose that $\kappa=\theta\left(\frac 1q-\frac
1p\right)_+$.
\begin{enumerate}
\item Let $p\ge q$ and $\alpha_0:=\alpha-(1-\gamma)\left(\frac 1q-\frac
1p\right)>0$ for $\kappa_w>\frac{\theta}{q}$,
$\alpha_0:=\alpha-1-(1-\gamma)\left(\frac 1q-\frac 1p\right)>0$
for $\kappa_w=\frac{\theta}{q}$. Then
$$
e_n(S_{u,w,{\cal A}}:l_p({\cal A}) \rightarrow l_q({\cal A}))
\underset{\mathfrak{Z}}{\asymp} (\log n)^{-\alpha_0} \rho(\log n)
\tau^{-\frac 1q+\frac 1p}(\log n).
$$
\item Let $p<q$, $\alpha_0:=\alpha>0$ for $\kappa_w>\frac{\theta}{q}$
and $\alpha_0:=\alpha-\frac 1q>0$ for $\kappa_w=\frac{\theta}{q}$.
In addition, suppose that $\alpha_0\ne \frac 1p-\frac 1q$. Then
$$
e_n(S_{u,w,{\cal A}}:l_p({\cal A}) \rightarrow l_q({\cal A}))
\underset{\mathfrak{Z}}{\asymp} n^{\frac 1q-\frac 1p} (\log
n)^{-\alpha_0-\frac 1q+\frac 1p} \rho(\log n)
$$
if $\alpha_0>\frac 1p-\frac 1q$,
$$
e_n(S_{u,w,{\cal A}}:l_p({\cal A}) \rightarrow l_q({\cal A}))
\underset{\mathfrak{Z}}{\asymp} n^{-\alpha_0} \rho(n)
$$
if $\alpha_0<\frac 1p-\frac 1q$.
\end{enumerate}
\end{enumerate}
\end{Trm}

If $\theta=0$, $\kappa=0$, then we suppose that $\rho_u(t)=|\log
(t+2)|^{-\lambda_u}$, $\rho_w=|\log (t+2)|^{-\lambda_w}$,
$\tau(t)=|\log (t+2)|^{\nu}$. Denote $\lambda
=\lambda_u+\lambda_w$. If $\kappa_w>0$,
$\alpha-(1-\gamma)\left(\frac 1q-\frac 1p\right)_+>0$, then the
bounds for $C(j)$ from the sharp two-sided estimate of
$\mathfrak{S}^{p,q}_{{\cal A}_{\eta_{j,i}},u,w}$ are already
obtained. If $\kappa_w>0$, $\alpha-(1-\gamma)\left(\frac 1q-\frac
1p\right)_+=0$, $\lambda>(1-\nu)\left(\frac 1q-\frac 1p\right)_+$,
then $C(j)$ for $j\ge 2$ is defined as follows. If $p\le q$, then
$C(j)=[\log(m_*j)]^{-\lambda}$; if $p>q$, then
$$C(j)=(m_*j)^{\gamma\left(\frac 1q-\frac 1p\right)} [\log(m_*j)]
^{-\lambda+\frac 1q-\frac 1p}.$$ It follows from estimates in
\cite{vas_vl_raspr2}, \cite{vas_har_tree}.

\begin{Trm} \label{dis_t2}
Suppose that $\theta=0$, $\kappa=0$, $\kappa_w>0$.
\begin{enumerate}
\item Let $\alpha-(1-\gamma)\left(\frac 1q-\frac 1p\right)_+>0$.
Then
$$
e_n(S_{u,w,{\cal A}}:l_p({\cal A}) \rightarrow l_q({\cal A}))
\underset{\mathfrak{Z}}{\asymp} n^{-\frac{\alpha}{1-\gamma}+\frac
1q-\frac 1p} (\log n)^{-\lambda-\frac{\alpha\nu}{1-\gamma}}.
$$
\item Let $\alpha-(1-\gamma)\left(\frac 1q-\frac 1p\right)_+=0$,
$\lambda>(1-\nu)\left(\frac 1q-\frac 1p\right)_+$.
\begin{enumerate}
\item Suppose that $p\ge q$. Then
$$
e_n(S_{u,w,{\cal A}}:l_p({\cal A}) \rightarrow l_q({\cal A}))
\underset{\mathfrak{Z}}{\asymp} (\log
n)^{-\lambda+(1-\nu)\left(\frac 1q-\frac 1p\right)}.
$$
\item Let $p<q$, $\lambda\ne \frac 1p-\frac 1q$. Then
$$
e_n(S_{u,w,{\cal A}}:l_p({\cal A}) \rightarrow l_q({\cal A}))
\underset{\mathfrak{Z}}{\asymp} n^{\frac 1q-\frac 1p} (\log
n)^{-\lambda+\frac 1p-\frac 1q}
$$
for $\lambda>\frac 1p-\frac 1p$,
$$
e_n(S_{u,w,{\cal A}}:l_p({\cal A}) \rightarrow l_q({\cal A}))
\underset{\mathfrak{Z}}{\asymp} n^{-\lambda}
$$
for $\lambda<\frac 1p-\frac 1q$.
\end{enumerate}
\end{enumerate}
\end{Trm}

\end{document}